\documentclass[11pt,a4paper]{amsart}

\usepackage[utf8]{inputenc}
\usepackage[T1]{fontenc}   % Police contenant les caractères français
\usepackage{geometry}     % Définir les marges
\usepackage[french, english]{babel} % Placez ici une liste de langues, la
\usepackage{amsmath}
\usepackage{amsfonts}
\usepackage{amsthm}

\usepackage{amscd}
\usepackage{hyperref}
\usepackage{xcolor}
\usepackage{enumitem}

\usepackage{graphicx}
\usepackage{caption}
\usepackage{subcaption}

\usepackage{pstricks,pst-plot}

  \usepackage{amssymb}
  \usepackage{verbatim}
  \newif\ifpdf
  \ifpdf
   \usepackage[pdftex]{graphicx}
   \usepackage[pdftex]{hyperref}
  \else
   \usepackage{graphicx}
  \fi

\usepackage[all]{xy}

\usepackage{enumitem}

\newcommand{\R}{\mathbb{R}}

\newcommand{\N}{\mathbb{N}}

\newcommand{\C}{\mathbb{C}}

\newcommand{\D}{\mathbb{D}}

\newcommand{\rcal}{\mathcal{R}}

\newcommand{\prw}{\mathrm{proj}_2}

\newtheorem{thm}{Theorem}[section]
\newtheorem{prop}[thm]{Proposition}
\newtheorem{coro}[thm]{Corollary}
\newtheorem{defi}[thm]{Definition}
\newtheorem{lem}[thm]{Lemma}

\newtheorem{rem}[thm]{Remark}

\newcommand{\id}{\mathrm{Id}}

\newcommand{\hfrak}{\mathfrak{h}}

\newcommand{\bcal}{\mathcal{B}}

\newcommand{\lcal}{\mathcal{L}}

\newcommand{\re}{\mathrm{Re}}
\newcommand{\im}{\mathrm{Im}}

\newcommand{\proj}{\mathrm{proj}}

\numberwithin{equation}{section}

\newcommand{\eps}{\epsilon}

\author{Luka Boc Thaler}
\thanks{The author was supported by the research program P1-0291 and projects N1-0237 and J1-70033, all from ARRS, Republic of Slovenia}

\address{L. Boc Thaler: Faculty of Mathematics and Physics, University of Ljubljana, SI--1000 Ljubljana, Slovenia. Institute of Mathematics, Physics and Mechanics, Jadranska 19, 1000 Ljubljana, Slovenia.} \email{luka.boc@fmf.uni-lj.si}

\title{Spiral Domains and Lavaurs-Type Renormalization for Parabolic Germs of $\C^2$}

\begin{document}

\maketitle

\begin{abstract}
We study the local dynamics of holomorphic germs $P:\mathbb C^2\to\mathbb C^2$ tangent to the identity whose 2-jet at the origin is $(J_0^2P)(z,w)= (z-z^2,w+w^2+bz^2)$.
We prove the existence of parabolic domains for all values of the parameter $b$, showing in particular that for $b>1/4$ there are spiral domains, i.e. parabolic domains whose orbits converge to the origin without being tangent to any fixed direction. We then establish a Lavaurs-type renormalization theorem for a class of non-skew-product maps, extending earlier results known in the skew-product case. As applications, we obtain new topological invariants for such germs and construct a Fatou component with both rank-one and rank-zero limit maps. We also give an example of a polynomial self-map of $\mathbb C^3$ with an elliptic fixed point admitting a wandering domain with non-contractible limit set.
\end{abstract}

\section{Introduction}
One of the central problems in local holomorphic dynamics is to describe the behavior of orbits near a fixed point. In this paper we consider discrete dynamical systems arising from the iteration of a holomorphic germ
$f:(\mathbb{C}^n,0)\to(\mathbb{C}^n,0)$ with $f(0)=0$, and we study the local dynamics in a neighborhood of the origin. A classical approach is to seek a holomorphic change of coordinates that conjugates $f$ to a simpler normal form. When $f$ is holomorphically linearizable, the local dynamics is completely determined by the linear part $df_0$. It is well known that the existence of such a linearization depends strongly on resonance and small-divisor phenomena, and in particular on the arithmetic properties of the eigenvalues of $df_0$, see \cite{abate,bracci,yoccoz}.

A fundamental situation in which one cannot, in general, obtain a linear normal form on a full neighborhood of the fixed point is the \emph{parabolic} case, namely when $f$ is \emph{tangent to the identity}:
$$
df_0=\mathrm{Id}\qquad\text{and}\qquad f\not\equiv \mathrm{Id}.
$$
In one complex dimension, the local dynamics of such germs is well understood, thanks to the celebrated \emph{Leau--Fatou flower theorem}. It asserts that there exist simply connected attracting and repelling petals, each having $0$ on its boundary, whose union covers a punctured neighborhood of the origin. Moreover, on each petal the germ (respectively, its local inverse) is holomorphically conjugate to a translation.

In higher dimension the dynamics of germs tangent to the identity is substantially more intricate, and a complete description on a neighborhood of the fixed point is known only in special settings. Significant progress has been obtained in dimension two, where partial analogues of the Leau--Fatou picture guarantee the existence of invariant one-dimensional objects and, in certain situations, two-dimensional stable manifolds, see \cite{hakim1998analytic,vivas2012,lopez2020stable,ecalleresurgentes}. The most general result for a germ $f$ tangent to the identity in $\mathbb{C}^2$ is due to Abate \cite{abate2001residual} which states that either $f$ admits a curve of fixed points through the origin, or else $f$ possesses \emph{parabolic curves}, i.e. one-dimensional complex manifolds with $0$ in their boundary that are invariant under $f$ and on which all forward orbits converge to the origin. In particular, such germs always exhibit nontrivial stable dynamics.

However, except in special cases, these results do not provide a description of the dynamics on a full neighborhood of the fixed point comparable to the one-dimensional flower theorem, and it is known that no general statement of this form can hold for all germs tangent to the identity. One such special case was studied in \cite{abate2011poincare}, where the authors show that, for the time-one map of the flow determined by a homogeneous vector field, the dynamics in a full neighborhood of the origin can be described by studying the geodesics of an associated meromorphic connection.

\medskip
A phenomenon closely related to parabolic dynamics is the \emph{parabolic implosion} \cite{Lavaurs, shi} which has many important consequences in complex dynamic of one variable. For example, it was used to prove the discontinuity of Julia sets near the parabolic parameter. The techniques of parabolic implosion have also been extended to higher dimensions, where they arise naturally in the study of semi-parabolic dynamics and germs tangent to the identity, see \cite{bedford2012parabolic, ALR}. In particular, they have been used to construct polynomial skew-products of $\C^2$ with wandering Fatou components, see \cite{ABDPR,astorg2019wandering,ABT}. Motivated by these results in the skew-product dynamics, the problem of non-autonomous parabolic implosion in  one variable was further studied in \cite{Viv20,HSV, AB}.

\medskip

In this paper, we continue the investigation of the local dynamics of holomorphic maps $P: \C^2 \to \C^2$, whose 2-jet at the origin is 
\begin{equation}\label{2jet}
(J_0^2P)(z,w)=(z-z^2,w+w^2+bz^2),
\end{equation}
which was initiated in \cite{ABT} in the special case of skew products, i.e. maps of the form $P(z,w)=(p(z), q(z,w))$. In that setting, it was shown that for $b\in(1/4,\infty)$ such maps always have parabolic domains where orbits converge non-tangentially to the origin. Furthermore, we proved the existence of a type of parabolic implosion, in which the renormalization limits are different from previously known cases. This had a number of consequences, including the construction of new topological invariants and the existence of wandering domains near the origin with rank-one limit maps. The aim of the present paper is to extend these results beyond the class of skew-products.

\subsection{Statement of results}
Let us first recall that a \emph{parabolic domain} of $P$ is a maximal invariant connected domain $\Omega\subset \mathbb{C}^2$ such that the origin belongs to the boundary of $\Omega$ and the iterates $P^n|_\Omega$ converge locally uniformly on $\Omega$ to the origin. Moreover, we say that a parabolic domain is tangent to a direction $v$ if every point in the domain is attracted to the origin along trajectories tangent to $v$. Invariant parabolic domains which are not tangent to any direction are also sometimes called \emph{spiral domains}. Such domains were first constructed by Rivi \cite[Proposition 4.4.4]{rivi} and later studied by Rong \cite[Theorem 1.4]{Rong}. 

We begin by discussing the existence of parabolic domains for maps whose 2-jet at the origin is equal to \eqref{2jet}, a question that depends only on the parameter $b$:

\begin{thm}\label{th:parbas}
	Let $P$ be a holomorphic endomorphism of $\mathbb{C}^2$ whose 2-jet at the origin is equal to \eqref{2jet}. Then,
	\begin{enumerate}
		\item if $b \in (\frac{1}{4},+\infty)$, the map $P$ has an invariant
		spiral domain.
		\item if $b\in\mathbb{C} \backslash (\frac{1}{4},+\infty)$, the map $P$ has an invariant parabolic domain $\Omega$ that is tangent to one of its non-degenerate characteristic directions. 
	\end{enumerate}
\end{thm}

Next, we focus on maps $P(z,w)=(p(z,w),q(z,w))$ whose $2$-jet at the origin is given by \eqref{2jet} and which additionally satisfy
$$
p(0,w)=0
\qquad\text{and}\qquad
\frac{\partial p}{\partial z}(0,w)=1.
$$
These conditions imply that the family of functions $p_w(z):=p(z,w)$ has a persistent parabolic fixed point at the origin, and they are needed for the application of parabolic renormalization techniques, see Remark \ref{zwk}. 
Up to conjugacy by an automorphism of $\mathbb{C}^2$, such maps can be reduced to the form
\begin{equation}\label{map1}
\left\{
\begin{array}{l}
p(z,w)= z-z^2+a_{3,0} z^3 +a_{2,1}z^2w+O(z^4,z^2w^2,z^3w), \\
q(z,w)= w+w^2+bz^2+b_{0,3} w^3+b_{3,0} z^3+O(\|(z,w)\|^4),
\end{array}
\right.
\end{equation}
where $a_{3,0}, a_{2,1}, b, b_{0,3}, b_{3,0} \in \mathbb{C}$.
\medskip

{\bf Notation:} {\it Throughout this paper, we will be using the notation $p(z,w)=p_0(z)+g(w)z^2+O(z^3w)$, where $g(w):=a_{2,1}w+O(w^2)$ is a holomorphic function, and let $q_z(w):=q(z,w)$. The origin is a parabolic fixed point of functions $p_0(z)$ and $q_0(w)$, and we denote their basin of attraction by $\mathcal{B}_{p_0}$ and $\mathcal{B}_{q_0}$, respectively.}
\medskip

Let $\theta:\mathcal{B}_{q_0}\rightarrow \C$ be a holomorphic function given by 
\begin{equation}\label{eq:1}
\theta(w):=\lim_{n\rightarrow\infty}\sum_{k=0}^{n-1}g(q_0^k(w))+a_{2,1}\ln{n}.
\end{equation}
This function is well defined because $q_0^k(w)=-\frac{1}{\phi_{q_0}^{\iota}(w)+k}+O(\ln k/k^2)$, where $\phi_{q_0}^{\iota}:\mathcal{B}_{q_0}\rightarrow \C$ is an \emph{incoming Fatou coordinate} of $q_0$, see \cite[Section 2]{ABT}.

\begin{defi} Let $P$ be of the form \eqref{map1}, and $\alpha,\sigma \in \C$. The generalized Lavaurs map of phase $\sigma$ and parameter $\alpha$ is defined as 
	\begin{equation}\label{glavaurs}
	\mathcal{L}(\alpha,\sigma; z,w):=\psi^o_{q_0}\left(\alpha\phi^{\iota}_{q_0}(w)+(1-\alpha)\phi^{\iota}_{p_0}(z)+(\alpha-1)\theta(w)+\sigma\right),
	\end{equation}
	where $\phi_{p_0}^\iota$ is the incoming Fatou coordinate of $p_0$, $\psi_{q_0}^o$ is the outgoing Fatou parametrization of $q_0$, and $\theta$ is given by \eqref{eq:1}.

	\end{defi}
	
Notice that, for $\alpha=1$, the map $w \mapsto \lcal(\alpha,\sigma; z,w)$ does not depend on $z$ and coincides with the classical Lavaurs map of phase $\sigma$ of the one-variable polynomial $q_0$.

\medskip

Finally, we introduce the following notation:

\begin{equation}\label{eq:c} 
	c:=\frac{\sqrt{4b-1}}{2}, \quad \alpha_0:=e^{\pi/c}, \quad \beta_0:=(b_{0,3}-a_{3,0}+a_{2,1})(\alpha_0-1),
\end{equation}
and observe that for $b>\frac{1}{4}$, we have $c>0$ and $\alpha_0>1$.

\medskip
The following is the main technical result of this paper.

\begin{thm}\label{th:maintech}
Let $P$ be a map of the form \eqref{map1} and let $h$ be a holomorphic function. Let $\alpha_0,\beta_0$ be as in \eqref{eq:c}, and assume that $b>\frac{1}{4}$ and $\beta_0 \in \R$. Let $M_n:= \lfloor(\alpha_0-1)n+ \beta_0\ln n \rfloor$ and $\rho_n:=\{(\alpha_0-1)n+ \beta_0\ln n \}$.
	\begin{align*}
		 P^{M_n}\left(p_0^n(z)+\frac{h(z,w)}{n^2}+o\left(\frac{1}{n^2}\right),w\right) = (p_0^{M_n+n}(z),\mathcal{L}(\alpha_0,\Gamma-\rho_n +(\alpha_0-1)h(z,w);z,w))+ o(1)
	\end{align*}
  
where the convergence is uniform on compact subsets of $\bcal_{p_0} \times \bcal_{q_0}$,
	and where $\Gamma$ is an explicit constant depending only on $a_{3,0}, a_{2,1},b,b_{0,3},b_{3,0}$.
	\end{thm}

The usefulness of results of this type is that, when applied successively, they can provide increasingly precise estimates for high iterates of $P$ in terms of the iterates of the maps $\lcal_z: w \mapsto \lcal(\alpha_0, \sigma; z,w)$. In this way, dynamical properties of $\lcal_z$ may be transferred to obtain information about the dynamics of $P$. In the case of skew products, this transfer is relatively straightforward because of the special form of the map, whereas in the present setting additional work is needed to justify such a transition.

\medskip 

Before stating some consequences of the main theorem, we introduce some additional terminology. Given real numbers $\alpha>1$ and $\beta\in\mathbb{R}$, we say that a strictly increasing sequence of positive integers $(n_k)_{k\geq 0}$ is \emph{$(\alpha,\beta)$-admissible} if its \emph{phase sequence} $(\sigma_k)_{k\geq 0}$, defined by $\sigma_k:= n_{k+1}-\alpha n_k-\beta\ln{n_k}$, is bounded.

A first consequence of Theorem \ref{th:maintech} is that, within our class of maps, the coefficient $b$ is a topological invariant.

\begin{coro}\label{top}
 	Let $P_1$ and $P_2$ be two maps of the form \eqref{map1} and assume that $g_i(w)\equiv 0$. Furthermore, assume that there exists a homeomorphism $\hfrak$ defined near the origin, with $\hfrak(0,0)=(0,0)$, such that 
 	$$\hfrak \circ P_1 = P_2 \circ \hfrak.$$
 	Let $b_i, \alpha_i, \beta_i$ (with $1 \leq i \leq 2$) be as in \eqref{eq:c}, and assume that 
 	$b_i>\frac{1}{4}$ and $\beta_i \in \R$. If both pairs $(\alpha_i, \beta_i)$ admit an $(\alpha_i, \beta_i)$-admissible sequence with a convergent phase sequence then $(\alpha_1, \beta_1)=(\alpha_2, \beta_2)$, and, in particular, $b_1=b_2$.

 \end{coro}

Another consequence of Theorem \ref{th:maintech} is the existence of Fatou components admitting limit maps of different ranks.

\begin{coro}\label{prop:wd} Let $P$ be a map of the form \eqref{map1} and assume that $g(w)\equiv 0$. Furthermore, let $b, \alpha_0, \beta_0$ be as in \eqref{eq:c}, and assume that $b>\frac{1}{4}$ and $\beta_0\in \R$. If there exists an $(\alpha_0,\beta_0)$-admissible sequence with a convergent phase sequence, then $P$ has a Fatou component with both rank one and rank zero limit maps. 
\end{coro}

Finally, in the last section we give an example of a polynomial self-map of $\C^3$ with an elliptic fixed point at the origin, namely 
$$
\tilde{F}(z,w,t):=
\left(\lambda (z+z^2t),\lambda(w+w^2t+\left(1+\frac{2\pi i}{\ln 2}\right)zwt),\lambda^{-1}t\right).
$$
where $|\lambda|=1$ is not a root of unity. This map has a wandering domain with a non-contractible limit set. 

\medskip
Although the organization of the paper and many of the underlying ideas closely parallel those of \cite{ABT}, the non-skew-product case is not a formal consequence of the skew-product one. The coupling between the two coordinates creates additional error terms and leads to several technical difficulties. These difficulties become especially relevant when one tries to transfer the renormalization results to statements about the actual dynamics of iterates of $P$.

\section{Parabolic domains}\label{sec:parbas}

Let $P$ be a holomorphic germ fixing the origin that is tangent to the identity of order $k \ge 2$, i.e. a map with a homogeneous expansion $P = \mathrm{Id} + P_k + P_{k+1} + \ldots$ where $P_k\not\equiv 0$. We say that $v \in \mathbb{C}^2$ is a \emph{characteristic direction} for $P$ if there exists a $\lambda \in \mathbb{C}$ so that
$P_k(v) = \lambda v$. 
%\begin{equation*}
%P_k(v) = \lambda v,
%\end{equation*}
If $\lambda \neq 0$, then $v$ is said to be \emph{non-degenerate} otherwise it is degenerate. The \emph{director} of a characteristic direction $v$ is an eigenvalue of a linear operator 
$$
d(P_k)_{[v]}-\id:T_{[v]}\mathbb{P}^1\rightarrow T_{[v]}\mathbb{P}^1.
$$

A \emph{parabolic curve} for $P$ is an injective holomorphic map $\varphi: \Delta\rightarrow\mathbb{C}^2$, satisfying the following properties:
\begin{enumerate}
\item $\Delta$ is a simply connected domain in $\C$ with $0\in\partial\Delta$
	\item $\varphi$ is continuous at the origin and $\varphi(0)=(0,0)$,
	\item $\varphi(\Delta)$ is invariant under $P$ and $P^n|_{\varphi(\Delta)}\rightarrow (0,0)$ uniformly on compact subsets.
\end{enumerate}
We say that a parabolic curve is tangent to $[v]\in \mathbb{P}^{1}$ if $[\varphi(\xi)]\rightarrow[v]$ as $\xi\rightarrow 0$ in $\Delta$. This implies that for any given point $z$ in the parabolic curve the orbit $(P^n(z))$ converges to the origin \emph{tangentially} to $v$, i.e. $[P^n(z)] \rightarrow [v]$ in $\mathbb{P}^1$. We now recall the following classical result due to Hakim \cite{hakim1994attracting, hakim1998analytic}:

\begin{thm}\label{thm:Hakim} Let $P: \mathbb{C}^2 \rightarrow \mathbb{C}^2$ be a holomorphic germ fixing the origin that is tangent to the identity of order $k \ge 2$. Then for any non-degenerate characteristic direction $v$ there exist (at least) $k-1$ parabolic curves for $P$ tangent to $[v]$. Moreover if the real part of the director of a non-degenerate characteristic direction $v$ is strictly positive, then there exists an invariant parabolic domain in which every point is attracted to the origin along a trajectory tangent to $v$.
\end{thm}

From now on, let $P$ be a map whose 2-jet at the origin is equal to \eqref{2jet} and observe that its characteristic directions are determined by the equations
$$
\left\{
\begin{array}{ll}
-z^2&=\lambda z \\
w^2 + bz^2&=\lambda w
\end{array}
\right.
$$

Then, aside from the trivial parabolic curve $z=0$ with non-degenerate characteristic direction $(0, 1)$, there are two parabolic curves $z \mapsto (z,\zeta^\pm(z))$
which are tangent to the non-degenerate characteristic directions $(1, c^\pm)$, where $c^\pm$ 
are the roots of
\begin{equation}\label{equ}
u^2 +u+b=0.
\end{equation} 

The new contribution of Theorem \ref{th:parbas} is contained in part (1), since part (2) follows from results of Hakim and Vivas, see \cite[Proposition 3.2]{ABT}. Therefore it suffices to prove the following proposition.

\begin{prop}\label{prop:b>1/4}
If $b>\frac{1}{4}$, then each of the two non-vertical parabolic curves is contained in spiral domains.
\end{prop}

\begin{proof} 
 Let $z\mapsto (z,\zeta^\pm(z))$ be one of the parabolic curves tangent to a non-degenerate characteristic direction $(1,c^\pm)$. Since it is invariant under $P$, it has to satisfy the equality $q_z(\zeta^\pm(z))=\zeta^\pm(p(z))$. A direct computation gives $\zeta^\pm(z):=c^{\pm}z+O(z^2)$. For $z$ close to the origin, we can define a change of coordinates $\psi^{\pm}(z)=(z,w+\zeta^{\pm}(z))$ which conjugates the map $P$ to a map of the form
\begin{equation}\label{conjug}
(z,w)\mapsto (z-z^2+O(z^3,z^2w, zw^2,w^3),w+w^2+2c^\pm zw +O(zw^2,z^2w,w^3)),
\end{equation}
where $c^\pm=-\frac{1}{2}\pm i\frac{x}{2}$ is the solution of the equation $u^2+u+b=0$ and $x>0$. Note that $(1,0)$ is now a non-degenerate characteristic direction of this map. For the rest of the proof, we focus on the case of $c^+$; in the case of $c^-$, the computations are identical with an appropriate change of sign.
\medskip

After the blow-up $w=uz$ of the map \eqref{conjug}, we obtain 
\begin{equation}\label{blowup}
\tilde{P}(z,u)=(z-z^2+O(z^3,z^3u),u(1+ ix z)+zu^2+ O(z^2u)).
\end{equation}
where the $O(z^2u)$ term is holomorphic on a neighborhood of the origin. Define $\mathbb{D}(r,r)=\{z\in\C: |z-r|<r\}$ and $D_r:=\{(z,u)\mid |u|<r, z\in \mathbb{D}(r,r)\}$ and assume that $r$ is sufficiently small so that $\tilde{p}_u(\mathbb{D}(r,r))\subset \mathbb{D}(r,r)$, where $\tilde{p}_u(z)=\tilde{p}(z,u)$.

\begin{lem}\label{lem:blowup}There exists a sequence of real numbers $0<r_j<r$ such that for any 
$(z_0,u_0)\in \mathcal{D}:=\bigcup_{j\geq 1} \{(z,u)\mid |u|<r_j, z\in \mathbb{D}(r,r\frac{j}{j+1})\}$ we have 
$\tilde{P}^n(z_0,u_0)\in D_r$ for all $n\geq 0$. Moreover, the sequence $\tilde{P}^n(z_0,u_0)$ is bounded away from the origin.
\end{lem}

\begin{proof}[Proof of Lemma \ref{lem:blowup}]
 First, observe that for sufficiently small $r>0$, there exists a holomorphic function $h(z)$ such that 
$$\tilde{P}(z,u)=(\tilde{p}(z,u),\tilde q(z,u))=(z-z^2+O(z^3,z^3u),ue^{ix z+z^2h(z)}+zu^2+ O(z^2u^2)).
$$

We need to prove that for any given $j \in \N^*$ there exists $0<r_j<r$ such that for every $(z_0,u_0) \in K_j\times \D(0,r_j)$ we have $(z_{n},u_{n}):= \tilde{P}^n(z_0,u_0)\in D_r$ for all $n\geq 1$. Clearly this holds for $n=1$ and $r_j$ is small enough.

Let us write $(z_n,u_n)=\tilde{P}^n(z_0,u_0)$. We prove by induction that for all $r_j$ sufficiently small we have 
\begin{enumerate}
  \item $z_n=\frac{1}{n}+O\left(\frac{\ln n}{n^2} \right)$,
  \item $|u_n|<r$.
\end{enumerate}
for all $n\geq 1$, with uniform bounds on $K_j\times \D(0,r)$.

Note that when $n=1$, this clearly holds if $r_j$ is sufficiently small. Now assume that these items hold for all $k\leq n$. Observe that
\begin{align*}
 \phi_{p}(\tilde{p}(z,u))&=\phi_{p}(p(z)+O(z^3u))=\phi_{p}(p(z))+O(\phi_{p}'(z)z^3u)\\
 &=\phi_{p}(z)+1+O(zu)
\end{align*}
or in other words $\phi_{p}(z_{k+1})=\phi_{p}(z_k)+1+O(z_k u_k)$, for all $k\leq n$ with uniform bounds on $K_j\times \D(0,r)$. Therefore we have
$$\phi_{p}(z_{n+1})=\phi_{p}(z_k)+n+\sum_{k=0}^nO(z_k u_k)=\phi_{p}(z_0)+n+1+O(\ln{(n+1)})$$ 
and hence $z_{n+1}=\frac{1}{n+1}+O\left(\frac{\ln {(n+1)}}{(n+1)^2} \right)$ which proves the first item.

For the second assertion, we have
\begin{align*}
  u_{n+1}&=u_{n}e^{ix z_n+z_n^2h(z_n)}+z_nu_n^2+ O(z_n^2u_n^2)\\
    &= u_ne^{\frac{ix}{n}+\Theta_n}+\frac{u_n^2}{n}+ O\left(\frac{u_n^2\ln n}{n^2}\right)
\end{align*}
where $\Theta_n =O\left(\frac{\ln n}{n^2}\right)$ and $O\left(\frac{u^2\ln n}{n^2}\right)$ with uniform bounds on $K_j\times \D(0,r)$.

Now define $U_n:=-\frac{1}{u_n}$ and observe that

$$
U_{n+1}=U_ne^{-\frac{ix}{n}-\Theta_n}+\frac{1}{n}+ O\left(\frac{\ln n}{n^2},\frac{\ln n}{U_n n^2}\right).
  $$ 

By our assumption $|U_0|>\frac{1}{r_j}$ and moreover $|U_k|>\frac{1}{r}$ for all $k \leq n$. We claim that $|U_{n+1}|>\frac{1}{r}$. 
   \medskip

Observe that since $x$ is real, there exists $\tilde{C}_j>0$ such that 
$$
\tilde{C}_j^{-1}<\left|e^{\sum_{\ell=1}^{k} \frac{ix}{\ell}+\Theta_\ell}\right|<\tilde{C}_j
$$ 
on $K_j\times \D(0,r)$ for all $k \leq n$. After the non-autonomous change of coordinates 
 $$
 Y_n=e^{\sum_{k=1}^{n} \frac{ix}{k}+\Theta_k}U_n,
 $$
 we obtain 
 
 \begin{align*}
 Y_{n+1}&=Y_n+\frac{1}{n}e^{\sum_{k=1}^{n} \frac{ix}{k}+\Theta_k}+ O\left(\frac{\ln n}{n^2},\frac{\ln n}{Y_nn^2}\right)\\
 &=Y_n+\frac{1}{n}e^{ix\ln n + \mathfrak{h}_n(z_0,u_0)}+ O\left(\frac{\ln n}{n^2},\frac{\ln n}{Y_n n^2}\right)
 \end{align*}
 where $\mathfrak{h}_n(z,u):=ix\gamma+\sum_{k=1}^{n} \Theta_k$ is a holomorphic function. Here we have used the fact that $\sum_{k=1}^{n} \frac{1}{k}=\gamma+\ln n +O(\frac{1}{n})$. Since the bounds in $\Theta_k$ are uniform on $K_j\times \D(0,r)$, the sequence of functions $\mathfrak{h}_n$ converges uniformly on $K_j\times \D(0,r_j)$ to some holomorphic function $\mathfrak{h}$ if our two items hold for each $n$.

Since $x\neq0$ is real, by Abel's summation formula there exists a constant $C>0$ such that 
$$
\left|\sum_{k=1}^n\frac{1}{k}e^{ix\ln k}\right|=\left|\sum_{k=1}^n k^{-(1-ix)}\right|<C
$$
for all $n\geq 1$. Summation by parts gives us

\begin{align*}
\sum_{k=1}^\ell\frac{1}{k}e^{ix\ln k}e^{\mathfrak{h}_k}&=e^{\mathfrak{h}_\ell}\sum_{k=1}^\ell\frac{1}{k}e^{ix\ln k}-
\sum_{k=1}^{\ell-1}(e^{\mathfrak{h}_{k+1}}-e^{\mathfrak{h}_{k}})\sum_{j=1}^k\frac{1}{j}e^{ix\ln j}\\
&=O(1)+ \sum_{k=1}^{\ell-1}O(\frac{\ln k}{k^2})=O(1)
\end{align*}
 and therefore $Y_{\ell}=Y_0+O(1)$ for all $\ell\leq n+1$, where the constant in $O(1)$ depends only on $K_j\times \D(0,r)$. 
 It follows that there is $A_j>0$ such that for all $r_j$ small enough and all $(z_0,u_0)\in K_j\times \D(0,r_j)$ we have 
 $$\tilde{C}_j^{-1}|U_0|-A_j<|U_k|<\tilde{C}_j|U_0|+A_j$$
and in particular
$$|U_k|>\frac{1}{r}$$
for all $k\leq n+1$. Note that we may choose $A_j$ independently of $n$. We have shown that $|u_{n+1}|<r$ assuming that $r_j$ was sufficiently small, which completes the induction step. 

The fact that $r_j$ can be chosen independently of $n$ follows from the uniform estimates on $K_j\times \D(0,r)$. Finally, notice that the proof also shows that the sequence $(u_n)_{n \geq 0}$ is bounded away from the origin which concludes the proof of Lemma \ref{lem:blowup}.
\end{proof}

We now return to the proof of Proposition \ref{prop:b>1/4}. 
Let $\Omega:=\{(z,zu)\mid (z,u)\in \mathcal{D}\}$: it is a connected open set whose boundary contains the origin and such that $P(\Omega)\cap \Omega\neq\emptyset$.
From Lemma \ref{lem:blowup}, it immediately follows that the iterates $P_{|\Omega}^n$ converge to the origin locally uniformly on $\Omega$, and is therefore contained in some invariant parabolic domain. It remains to prove that orbits of points converge non-tangentially to the origin in that parabolic domain. Indeed, let $(z_0,w_0)\in \Omega$ and $(z_{n},w_{n})=P^{n}(z_0,w_0)$ and observe that, since $z_n\neq0$, for all $n \in \N$ we have $[z_n:w_n]=[1:\frac{w_n}{z_n }]=[1:u_n]$, where $u_n\sim Ce^{ix\log{n}}$ for some $C\neq 0$. The proof of Lemma \ref{lem:blowup} shows that every limit map of the iterates $(\tilde{P}^n)$ on $\mathcal{D}$ is of the form $(z,u)\mapsto (0, \eta(z,u))$, where $\eta$ is a non-constant holomorphic function and $\frac{\partial \eta}{\partial u}\not\equiv 0$. %Therefore, there is no vector $v\in\mathbb{C}^2$ such that the sequence $[P^n(z,w)]$ would converge to $[v]$ in $\mathbb{P}^1$ for all $(z,w)\in \Omega$.
Thus there is no direction $v\in\mathbb P^1$ such that all points in the parabolic domain converge to the origin tangent to $v$.
\end{proof}

\section{The error functions}\label{sec:erf}

In this section we introduce and study the functions $\tilde A(z,w)$, $ A(z,w)$ and $A_0(z)$, which are among the main objects in our analysis. These functions measure, in suitable local coordinates, how far the dynamics deviates from a translation. Although many of the arguments are similar to those in \cite{ABT}, we include them here because the fact that $P$ is not a skew product introduces additional terms that require separate estimates.

\medskip

Let $P$ be a map of the form \eqref{map1}. Hakim's work implies that, aside from the trivial parabolic curve contained in the invariant line $z=0$, there are two parabolic curves
which are tangent to the non-degenerate characteristic directions $(1, c^\pm)$ where $c^\pm:=-\frac{1}{2}\pm ic$. These parabolic curves can be written as holomorphic graphs $z\mapsto (z, \zeta^{\pm}(z))$ over a small petal $\D(r,r):=\{z\in\C: |z-r|<r\}$. Since the parabolic curves are invariant under $P$, the functions $\zeta^{\pm}$ satisfy the functional equations 
\begin{equation}\label{eq:curve}
q(z,\zeta^{\pm}(z))=\zeta^{\pm}(p(z,\zeta^{\pm}(z))).
\end{equation}
From these equations one easily computes the first few terms of the formal power series to which they are asymptotic:
\begin{equation}
	\zeta^{\pm}(z):=c^{\pm}z+\left(c^{\pm} \Theta+\frac{a_{3,0}+(b-1)b_{0,3}- a_{2,1} }{2}\right)z^2+O(z^3),
\end{equation}
where 
\begin{equation}\label{eq:deftheta}
	\Theta:=b_{0,3} +\frac{a_{3,0}-b_{0,3}+b_{3,0} +(b-1) a_{2,1} }{2b}.
\end{equation}

\begin{defi}
	Let
	$$
	\psi_z (w):=\frac{1}{2ic}\log \left(\frac{\zeta^+(z)-w}{w-\zeta^-(z)} \right) 
	$$
	where $\log$ is the principal branch of logarithm and let
	$$
	\psi_z^{\iota/o} (w):=\psi_z (w) \pm \frac{\pi}{2c}.
	$$
\end{defi}

With this choice of branch, $\psi_z$ is defined on $\C \backslash L_z$,
where $L_z$ is the real line through $\zeta^+(z)$ and $\zeta^-(z)$ minus the segment
$[\zeta^-(z), \zeta^+(z)]$. In particular, $\psi_z^\iota$ and $\psi_z^o$ are both
defined in a disk centered at $w=\frac{1}{2}(\zeta^+(z)+\zeta^-(z))$ whose radius is of order $z$.

\begin{defi}
	Let
	\begin{enumerate}
		\item $A(z,w):=\psi_{ p(z,w)}^{\iota/o} \circ q_{z}(w)-\psi_z^{\iota/o}(w)-z$
		\item $A_0(w):=-\frac{1}{q_0(w)}+\frac{1}{w}-1$
	\end{enumerate}
\end{defi}
The formula for $A(z,w)$ does not depend on whether the ingoing or outgoing coordinate $\psi_z$ is used, and is therefore well defined. The map $A$ is initially defined on the open set 
	$\left\{(z,w) \in V_r : w \notin L_z \text{ and } q_z(w) \notin L_{p(z,w)} \right\} $, where 
	$$
	V_r:=\D(r,r)\times\D(0,r),
	$$ 
	for some sufficiently small $r$.
	
\begin{prop}\label{prop:phianalytic}
	We have:
	\begin{enumerate}
		\item $A_0$ is analytic near $0$, and $A_0(w)=(b_{0,3}-1)w+O(w^2)$.
		\item There exists $r>0$ such that $A(z, \cdot)$ extends analytically to $V_r$.
	\end{enumerate}
\end{prop}

\begin{proof} Item (1) is a straightforward computation. For (2), observe that if $r>0$ is small enough, then $P(V_r)\subset V_{2r}$. Moreover for $(z,w) \in V_r$ such that $w \notin L_z$, $q_z(w) \notin L_{p(z,w)}$, we have:
	\begin{align*}
		A(z,w)&=\frac{1}{2ic}\log \left(\frac{q_z(w)-\zeta^+(p(z,w))}{q_z(w)-\zeta^-(p(z,w))} \right) -\frac{1}{2ic}\log \left(\frac{w-\zeta^+(z)}{w-\zeta^-(z)} \right) -z\\
		&=\frac{1}{2ic}\log \left(\frac{q_z(w)-\zeta^+(p(z,w))}{w-\zeta^+(z)} : \frac{q_z(w)-\zeta^-(p(z,w))}{w-\zeta^-(z)} \right)-z\\
		%&=\frac{1}{2ic}\log \left(\frac{q_z(w)-q_z(\zeta^+(z))}{w-\zeta^+(z)} : \frac{q_z(w)-q_z(\zeta^-(z))}{w-\zeta^-(z)} \right)-z.
	\end{align*}
We now consider the expression 
$$B_+(z,w)=\frac{q_z(w)-\zeta^+(p(z,w))}{w-\zeta^+(z)}$$
and show that $B_+(z,w)=1+O(z,w)$ is holomorphic. Observe that 
\begin{align*}
  B_+(z,w)&=\frac{q_z(w)-q_z(\zeta^+(z))}{w-\zeta^+(z)}+\frac{q_z(\zeta^+(z))-\zeta^+(p(z,w))}{w-\zeta^+(z)}\\
  &=\frac{q_z(w)-q_z(\zeta^+(z))}{w-\zeta^+(z)}+\frac{\zeta^+(p(z,\zeta^+(z)))-\zeta^+(p(z,w))}{w-\zeta^+(z)}\\
  &=\frac{q_z(w)-q_z(\zeta^+(z))}{w-\zeta^+(z)}-\frac{h_z(w)-h_z(\zeta^+(z))}{w-\zeta^+(z)}
\end{align*}
where in the second line we have used \eqref{eq:curve} and where $h_z(.)=\zeta^+(p(z,.))$, which is well defined on $\mathbb{D}(r,r)\times \mathbb{D}(0,r)$ for small $r$ since $p(z,w)=z-z^2+O(z^3,z^2w)$.
A direct computation yields 
$\partial_w h_z(w)=(\zeta^+)'(p(z,w))\partial_w p(z,w)=O(z^2)$ and hence 
$$h_z(w)-h_z(\zeta^+(z))=O\left(z^2(w-\zeta^+(z))\right).$$
On the other hand,
\begin{align*}
q_z(w)-q_z(\zeta^+(z))&=(\partial_wq_z)(\zeta^+(z))(w-\zeta^+(z))+O\left((w-\zeta^+(z))^2\right)\\
&=(1+2\zeta^+(z)+\ldots)(w-\zeta^+(z))+O\left((w-\zeta^+(z))^2\right).
\end{align*}
Using these two equations in the computation of $B_+(z,w)$ gives the desired result. The same argument shows that 
$$B_-(z,w)=\frac{q_z(w)-\zeta^-(p(z,w))}{w-\zeta^-(z)}=1+O(z,w)$$
by replacing the sign $+$ with $-$.

Therefore $A(z,w)=\log\left(\frac{B_+(z,w)}{B_-(z,w)}\right)-z=\log\left(\frac{1+O(z,w)}{1+O(z,w)}\right)-z$ has removable singularities at $w=\zeta^{\pm}(z)$.

\end{proof}

\begin{prop}\label{prop:estimateA} Let $\Theta$ be as in \eqref{eq:deftheta}, and let $0<s<r$. Then
	$$
	A(z,w)= z A_0(w) +\left(\Theta+\frac{1}{2}-b_{0,3}-a_{2,1}\right) z^2+	O(z^3, z^2 w),
	$$
	where the constants in the $O$-term are uniform on $V_s$.

\end{prop}

\begin{proof} Let $K$ be a compact in $V_r\backslash\{w=0\}$. A straightforward computation gives 
	\begin{align}
		\frac{1}{2ic}\log\left(\frac{w-\zeta^+(z)}{w-\zeta^-(z)} \right)&=\frac{1}{2ic}\left( \frac{\zeta^-(z)-\zeta^+(z)}{w}-\frac{(\zeta^+(z))^2-(\zeta^-(z))^2}{2w^2}\right)+O(z^3)\nonumber\\
		&=-\frac{z}{w}-\frac{\Theta z^2}{w}+\frac{z^2}{2w^2}+O(z^3).\label{eq:asymppsi}
	\end{align}
	Using this, we obtain that
	\begin{align*}
		\psi_{ p(z,w)}^{\iota/o} \circ q_{z}(w)&=-\frac{ p(z,w)}{ q_{z}(w)}-\frac{\Theta(p(z,w))^2}{ q_{z}(w)}+\frac{( p(z,w))^2}{2( q_{z}(w))^2}+O(z^3)\\
		&=-\frac{ z-z^2+a_{2,1}z^2w}{ q_{0}(w)}- \frac{\Theta z^2}{q_{0}(w)}+\frac{ z^2}{2(q_{0}(w))^2}+O(z^3,z^2w).
	\end{align*}
	This implies that
	\begin{align*}
		A(z,w)&=zA_0(w)+\Theta z^2\left(\frac{1}{w}-\frac{1}{q_{0}(w)}\right)\\
		&+\frac{z^2}{2}\left(\frac{1}{(q_{0}(w))^2}-\frac{1}{w^2}+\frac{2-2a_{2,1}w}{q_0(w)}\right)+O(z^3,z^2w)\\
		&=zA_0(w)+\Theta z^2+\frac{z^2}{2}(1-2b_{0,3}-2a_{2,1})+O(z^3,z^2w)\\
		&= z A_0(w) +\left(\Theta+\frac{1}{2}- b_{0,3}-a_{2,1}\right) z^2+	O(z^3, z^2 w).
	\end{align*}
	Here we used the fact that $A(z,\cdot)$ is analytic, hence all terms in $w$ with negative powers are cancelled. 
	
Finally, we note that the constant in the $O(z^3, z^2 w)$ a priori depends on $K \subset V_r\backslash\{w=0\}$.
By Proposition \ref{prop:phianalytic}, the function $\phi(z,w):=A(z,w)-z A_0(w)-\left(\Theta+\frac{1}{2}- b_{0,3}-a_{2,1}\right) z^2$ is holomorphic on $V_r$ gives the desired uniformity on $V_{s}$.
	
\end{proof}

\begin{defi}\label{defi:R}
	Let $\nu\in(\frac{1}{2},\frac{2}{3})$ and let 
	$$\mathcal{R}_z:=\{W \in \C : \frac{|z|^{1-\nu}}{10}<\re(W)<\frac{\pi}{c}-\frac{|z|^{1-\nu}}{10} \text{ and } -\frac{1}{2} < \im(W)<\frac{1}{2} \}$$
\end{defi}

\begin{defi}
	Let $\chi_z(W)=W+ ( b_{0,3}-1) R(z,W)$, where 
	\begin{equation}
		R(z,W):=cz e^{W} F_c(W)
	\end{equation}
	and $F_c$ is the primitive on $\rcal_0$ of $W \mapsto e^{-W} \cot(cW)$ vanishing at $\frac{\pi}{2c}$.
\end{defi}

A straightforward computation shows that 	$R(z,W)$ is a solution of the linear PDE
\begin{equation}
	-z \frac{\partial R}{\partial z} + \frac{\partial R}{\partial W} =c z \cot(cW).
\end{equation}

\begin{lem}\label{lem:invphi}
	We have
	$$(\psi_{z}^{\iota/o})^{-1}(W)=-c z \cot(cW) - \frac{z}{2}+{O(z^2 \cot(cW),z^2)} $$
\end{lem}

\begin{proof}
	We have:
	\begin{equation}
		(\psi_z^{\iota/o})^{-1}(W)=\frac{\zeta^+(z)-\zeta^-(z) e^{2icW}}{1- e^{2icW}}
	\end{equation}
	and using the asymptotic expansion of $ \zeta^{\pm}(z)$, we find
	\begin{equation}
		(\psi_z^{\iota/o})^{-1}(W)=-cz \cot(cW)-\frac{z}{2}+{O(z^2 \cot(cW),z^2)}.
	\end{equation}
\end{proof}

\begin{lem}\label{lem:dR}
	Assume that $\psi_z^\iota(w) \in \rcal_z$, and let $W:=\psi_z^\iota(w)$, $W_1:=\psi_{p(z,w)}^\iota \circ q_z(w)$ and $z_1:=p(z,w)$. Then 
	\begin{equation}
		\left|R(z_1, W_1)-R(z,W) - cz^2 \cot(cW) \right| = O\left(|z|^{2+\delta}\right).
	\end{equation}
 with $\delta:=2\nu -1>0$.
\end{lem}
\begin{proof}
Let $\hat{z}_1=p(z,0)$ and $\hat{W}_1=\psi_{p(z,0)}^\iota \circ q_z(w)$ Then, by \cite[Lemma 4.9]{ABT} we have
$$\left|R(\hat{z}_1, \hat{W}_1)-R(z,W) - cz^2 \cot(cW) \right| = O\left(|z|^{2+\delta}\right).$$
It remains to estimate the error $R(\hat{z}_1, \hat{W}_1)-R(z_1,W_1).$
Observe that on $\rcal_z$ we have $\partial_z R(z,W)=ce^WF_c(W)= O(\log\frac{1}{|z|})$ and $\partial_W R(z,W)=O(z^{\nu})$. Moreover, $\hat{z}_1-z_1=O(|z|^{2+\nu})$ and $\hat{W}_1-W_1=O(|z|^{1+\nu})$ and therefore
$$R(\hat{z}_1, \hat{W}_1)-R(z_1,W_1)=O(\partial_z R(z,W)(\hat{z}_1-z_1) +\partial_W R(z,W)(\hat{W}_1-W_1))=O(|z|^{1+2\nu})=O(|z|^{2+\delta})$$

\end{proof}

\begin{defi} We define $\tilde A(z,w) :=\chi_{p(z,w)} \circ \psi_{p(z,w)}^\iota \circ q_{z}(w) - \chi_z \circ \psi_z^\iota(w)-z $.
\end{defi}

\begin{prop}[Almost translation property]\label{prop:almosttransv2}
 Let $\delta:=2\nu-1>0$. Then
	$$|\tilde A(z,w)-\Lambda z^2|=O\left(|z|^{2+\delta}\right)$$
	for all $(z,w)$ such that $\psi_z^\iota(w) \in \mathcal{R}_z$, where $\Lambda :=\Theta+1-\frac{3 b_{0,3}}{2}-a_{2,1}$
\end{prop}

\begin{proof} Let $z_1:=p(z,w)$, $W:=\psi_z^\iota(w)$ and $W_1:=\psi_{z_1}^\iota \circ q_z(w)$. 
	We have
	\begin{align*}
		\tilde A(z,w) &=\chi_{z_1} \circ \psi_{z_1}^\iota \circ q_{z}(w) - \chi_z \circ \psi_z^\iota(w)-z \\
		&=\psi_{z_1}^\iota \circ q_{z}(w) - \psi_z^\iota(w)-z + ( b_{0,3}-1) (R(z_1, W_1) - R(z,W)).
	\end{align*}
	By Lemma \ref{lem:dR}
	\begin{align*}
		|\tilde A(z,w)-A(z,w) - cz^2 ( b_{0,3}-1) \cot(cW)|= O\left(|z|^{2+\delta}\right)
	\end{align*}

	On the other hand, by Proposition \ref{prop:estimateA} we have
	\begin{align*}
		A(z,w)&=z A_0(w) + \left(\Theta+\frac{1}{2}- b_{0,3}-a_{2,1}\right)z^2 + O(z^2 w,z^3)\\
		&=( b_{0,3}-1)zw+ \left(\Theta+\frac{1}{2}- b_{0,3}-a_{2,1}\right)z^2 + O(zw^2,z^2 w,z^3)\\
	\end{align*}
	so using Lemma \ref{lem:invphi}:
	
	$$
	A(z,w) =(1- b_{0,3}) cz^2 \cot(cW) + z^2 \left(\Theta+1-\frac{3 b_{0,3}}{2}-a_{2,1} \right)+ O\left(zw^2,z^2 w,z^3,z^3 \cot(cW)\right) 
	$$
	
	Combining these estimates and using that $z^3 \cot(cW)=O(|z|^{2+\nu})=O(|z|^{2+\delta})$ on $\mathcal{R}_z$, we get:
	\begin{equation}\label{eq:tildeA}
		|\tilde A(z,w) - \Lambda z^2| =O\left(|zw^2|,|z^2 w|,|z|^3,|z|^{2+\delta}\right) 
	\end{equation}
	Finally, since by assumption $\psi_z(w) \in \mathcal{R}_z$, we have 
	$|w| = O\left(|z|^{\nu}\right)$. Indeed, by Lemma \ref{lem:invphi} and the definition of $\rcal_z$,
		$w = O( z \sin(cW)^{-1})= O\left( \frac{|z|}{|z|^{1-\nu}}\right) = O(|z|^{\nu}) $.

		Using $\nu\in(\frac{1}{2},\frac{2}{3})$:
		\begin{itemize}
			\item $|zw^2| = O\left(|z|^{1+2\nu}\right) = O(|z|^{2+\delta})$, since $\delta=2\nu-1$ by definition;
			\item $|z^2w|=O\left(|z|^{2+\nu} \right) = O(|z|^{2+\delta}) $ since $2+\nu > 1+2\nu$;
			\item $|z^3|=O\left(|z|^{2+\delta}\right)$ (again, since $1+2\nu <3$).
		\end{itemize}
		Therefore \eqref{eq:tildeA} gives the required estimate 
		$$|\tilde A(z,w)-\Lambda z^2|=O\left(|z|^{2+\delta}\right).$$
\end{proof}

The proof of the following lemma can be found in \cite[Lemma 4.12]{ABT}.
\begin{lem}\label{lem:DLFc}
	As $W \to 0$ in $\rcal_0$, we have
	\begin{equation}
		F_c(W)=\frac{1}{c} \log (cW) - \frac{1}{c}\int_0^{\frac{\pi}{2c}} e^{-u} \log \sin(cu) du + o(1)
	\end{equation}
	Similarly, as $W \to \frac{\pi}{c}$ in $\rcal_0$, we have:
	\begin{equation}
		F_c(W) =e^{-\frac{\pi}{c}}\frac{1}{c} \log\left(\pi - cW\right) + \frac{1}{c} \int_{\frac{\pi}{2c}}^{\frac{\pi}{c}} e^{-u} \log \sin(cu) du + o(1)
	\end{equation}
\end{lem}

\begin{lem}\label{lem:invcomp}
	For $W\in\mathcal{R}_z$ and $\delta=2\nu-1>0$ we have
	$$(\chi_z\circ \psi_{z}^{\iota})^{-1}(W)=(\psi_{z}^{\iota})^{-1}(W)+{O( |z|^{1+\delta}\log\frac{1}{|z|})} $$
\end{lem}

\begin{proof}
First, observe that for $W\in \mathcal{R}_z$ we have $R(z,W)=O(|z|\log\frac{1}{|z|})$, which implies that $\chi_z^{-1}(W)=W+O(|z|\log\frac{1}{|z|})$. Then 
$$(\chi_z\circ \psi_{z}^{\iota})^{-1}(W)=(\psi_{z}^{\iota})^{-1}(W+O(R(z,W)))=(\psi_{z}^{\iota})^{-1}(W)+O\left( \frac{\partial (\psi_{z}^{\iota})^{-1}}{\partial W}(W)R(z,W)\right).$$
Now observe that for $W\in \mathcal{R}_z$ 
$$\frac{\partial(\psi_{z}^{\iota})^{-1}}{\partial W}(W)=O\left(\frac{z}{\sin^2(cW)}\right)=O\left(|z|^{2\nu-1}\right),$$
and therefore $O\left( \frac{\partial (\psi_{z}^{\iota})^{-1}}{\partial W}(W)R(z,W)\right)=O( |z|^{1+\delta}\log \frac{1}{|z|})$

\end{proof}

\section{Proof of the main theorem}

Fix $\frac{1}{2}<\nu<\frac{2}{3}$, define $k_n:=\lfloor n^\nu \rfloor$
 and denote $(z_j, w_j):=P^j(z_0,w_0)$.
%Let $K_n\times K\subset \cP^+_{\frac{r}{n}}\times \cP^-_r$ a compact for some $n>n_0$, where the choice of $n_0$ depends on the choice of the compact set $K\subset\D(-r,r)$. Finally for $(z_0,w_0)\in K_n\times K$ we let $(z_j, w_j):=P^j(z_0,w_0)$.

\medskip

{\bf Choice of constants:} We choose constants $\nu \in (\frac{1}{2},\frac{2}{3})$ and $r>0$ such that:
\begin{enumerate}[label=(\roman*)]
\item Proposition \ref{prop:phianalytic} applies. 
	\item The inverse $q_0^{-1}$ is well defined on $\mathbb{D}(0,2r)$. 
	\item $q_0(\D(-r,r))\subset\D(-r,r)$ and $q_0^{-1}(\D(r,r))\subset\D(r,r)$
%\item $|\sum_{j=0}^ng(q^j_0(w_0))|<\ln(2n)$ for all $n\geq1$ and all $w_0\in \D(-r,r)$
	\item $p_0(\D(r,r))\subset\D(r,r)$ and $p_0^{-1}(\D(-r,r))\subset\D(-r,r)$

\end{enumerate}

\medskip

{\bf Choice of compact set:} For the rest of this section, we fix a compact subset $K'\times K \subset \mathcal{B}_{p_0}\times \mathcal{B}_{q_0}$.
\medskip

{\bf Choice of integer:} Choose $n_0 \in \N$ large enough so that for every $n>n_0$ we have $p_0^n(K')\subset \D(r,r)$ and $q_0^{k_n}(K) \subset \D(-r,r)$, where $k_n:=\lfloor n^\nu \rfloor$.
\medskip

{\bf Notation:} Given a point $(z_0,w_0)\in K'\times K$, we write 
$$\eps_j:=\proj_1 P^j\left(p_0^n(z_0)+\frac{h(z_0,w_0)}{n^2}+o\left(\frac{1}{n^2}\right),w_0\right)$$
and $w_j:=q_{\eps_{j-1}} \circ q_{\eps_{j-2}} \circ \ldots \circ q_{\eps_0}(w_0)$.

\begin{rem}Unless otherwise stated, all the constants appearing in estimates depend only on the compact $K'\times K$, but not on the point $(z_0,w_0)$ nor on the integer $n$.
\end{rem}

\subsection{Entering the eggbeater} We begin with a lemma showing that the first $k_n$ elements of the non-autonomous orbit $\{w_j\}$ are suitably close to the autonomous orbit $\{q_0^j(w_0)\}$, provided that $n$ is sufficiently large.

\begin{lem}\label{lem:init}
	We have
	$$\eps_{k_n}=\left(n+k_n+\phi_{p_0}^\iota(z_{0}) +(1-a_{3,0})\ln{n}-h(z_0,w_0)-\sum_{j=0}^{k_n-1} g(q^j_0(w_0))+o(1)\right)^{-1}$$	and for $0\leq j\leq k_n$
	$$\phi_{q_0}^\iota(w_j)=\phi_{q_0}^\iota(w_0)+j+o(1).$$
	
	In particular we have $w_{k_n}=-\frac{1}{k_n}+O(\frac{\ln n}{k_n^2}) \in \D(-r,r)$.
\end{lem}

\begin{proof}
First, observe that for $(\eps_j, w_j)\in \D(r,r) \times \mathcal{B}_{q_0} \subset \mathcal{B}_{p_0}\times \mathcal{B}_{q_0}$ we have

\begin{equation*}
	\begin{array}{l}
		\phi_{p_0}^\iota(\eps_{j+1})=\phi_{p_0}^\iota(p_0(\eps_{j})+g(w_j)\eps_j^2+O(\eps^3_jw_j))=\phi_{p_0}^\iota(\eps_{j})+1-g(w_j)+O(\eps_jw_j) \\
  \phi_{q_0}^\iota(w_{j+1})=\phi_{q_0}^\iota(q_0(w_{j})+b\eps^2_j+O(\eps_j^3,\eps_jw_j^3,\eps_j^2w_j^2))=\phi_{q_0}^\iota(w_{j})+1+O(\frac{\eps_j^2}{w_j^2},\eps_jw_j)
	\end{array}
\end{equation*}

We prove the lemma by induction $j$. Let $(\eps_0, w_0)\in \D(r,r) \times K $, $0<\ell \leq k_n$ and $n$ sufficiently large. Assume that $(\eps_j, w_j)\in \D(r,r) \times \mathcal{B}_{q_0} $, $\eps_j=\frac{1}{n}+O(\frac{\ln{n}}{n^2})$ and $\phi_{q_0}^\iota(w_{j})=\phi_{q_0}^\iota(w_{0})+j+o(1)$ for all $0\leq j< \ell $, which is trivial when $\ell=1$. Then 

\begin{equation*}
	\begin{array}{l}
		\phi_{p_0}^\iota(\eps_{\ell})=\phi_{p_0}^\iota(z_{0})+n+\ell-h(z_0,w_0)-\sum_{j=0}^{\ell-1}g(w_j)+O(\frac{\ln{n}}{n}) \\
  \phi_{q_0}^\iota(w_{\ell})=\phi_{q_0}^\iota(w_{0})+\ell +O(\frac{\ell^3}{n^2},\frac{\ln{n}}{n})
	\end{array}.
\end{equation*}

Furthermore since $w_j=q_0^j(w_0)+O(\frac{j}{n^2},\frac{\ln{n}}{n j^2})$ and $\ell\leq k_n$ we get 

\begin{equation*}
	\begin{array}{l}
		\phi_{p_0}^\iota(\eps_{\ell})=\phi_{p_0}^\iota(z_{0})+n+\ell-h(z_0,w_0)-\sum_{j=0}^{\ell-1} g(q^j_0(w_0))+o(1) \\
  \phi_{q_0}^\iota(w_{\ell})=\phi_{q_0}^\iota(w_{0})+\ell +o(1)
	\end{array}.
\end{equation*}

For all sufficiently large $n$ and sufficiently small $r$, independently of $\ell$, we also get $(\eps_\ell, w_\ell)\in \D(r,r) \times \mathcal{B}_{q_0} $, and the conclusion follows from $\ell=k_n$.

\end{proof}

\begin{defi}[Approximate Fatou coordinate] Let $\Phi_\eps:= \chi_{\eps} \circ \psi_{\eps}^\iota$.
\end{defi}

The proofs of the following two lemmas follow from the same computations as in \cite[Lemma 5.3 \& Lemma 5.6]{ABT}.

\begin{lem}\label{prop:error1}
For all sufficiently large $n$	we have
	\begin{align*}
		\psi_{\eps_{k_n}}^{\iota}(w_{k_n})
		&=-\frac{\eps_{k_n}}{w_{k_n}}+{\frac{\eps_{k_n}^2}{2w_{k_n}^2}}+o\left(\eps_{k_n}\right).
	\end{align*}
\end{lem}

\begin{lem}[Comparison with incoming Fatou coordinates]\label{lem:comp}
	We have 
	\begin{equation*}
		\frac{1}{\eps_{k_n}}\Phi_{\eps_{k_n}}(w_{k_n})=\phi^\iota_{{q_0}}(w_{k_n}) +{\frac{k_n^2}{2n}}+ (1- b_{0,3}) \ln n + E^\iota + o(1)
		%		( b_{0,3}-1) (\theta_1(w_0)+\ln n) + 
		%		( b_{0,3}-1) \left(\ln c - \frac{1}{c}\int_0^{\frac{\pi}{2c}} e^{-u} \ln \sin(cu) du  \right) + o(1)
	\end{equation*}
	where $E^\iota:=( b_{0,3}-1) \left(\ln c - \int_0^{\frac{\pi}{2c}} e^{-u} \ln \sin(cu) du  \right)$.
\end{lem}

\begin{rem}\label{rem:initial} Note that 
$$w_{k_n}=-|\eps_{k_n}|^{\nu}+o(\eps_{k_n}^{\nu}),$$
Therefore, given a compact set $K\subset\mathcal{B}_{q_0}$, Lemma \ref{lem:init} and Lemma \ref{prop:error1} imply that $\psi^{\iota}_{\eps_{k_n}}(w_{k_n})\in\mathcal{R}_{\eps_{k_n}}$
for all $(z_0,w_0)\in K'\times K$ and all sufficiently large $n$. Moreover by Lemma \ref{lem:DLFc} $ \chi_{\eps_{k_n}}(W)=W+O(|\eps_{k_n}|\log\frac{1}{|\eps_{k_n}|})$ the same conclusion holds for $\Phi_{\eps_{k_n}}(w_{k_n})\in\mathcal{R}_{\eps_{k_n}}$.
\end{rem}

\subsection{Passing through the eggbeater}

Recall the following notations:
$$ 
c:=\frac{\sqrt{4b-1}}{2},\, \alpha_0:= e^{\pi/c}, \, \beta_0:= (b_{0,3}-a_{3,0}+a_{2,1} )(\alpha_0-1).$$
where $b > \frac{1}{4}$.

\begin{defi}\label{defi:mn} Let 
	$M_n:= \lfloor(\alpha_0-1)n+ \beta_0\ln n \rfloor$
	where $\lfloor \cdot\rfloor$ is the floor function. Let $\ell_n:= \lfloor e^{\frac{\pi}{c}}k_n\rfloor$ and $\rho_n:=\{(\alpha_0-1)n+ \beta_0\ln n \}$, where $\{\cdot\}$ denotes the fractional part. Finally, define $W_{j}:=\Phi_{\eps_{j}}(w_{j})$.  
	
\end{defi}

\begin{rem}\label{zwk}
The following lemma is where the assumption that the first coordinate of the map \eqref{map1} contains no terms of the form $zw^k$ ($k\geq 1$) is needed. If such terms were present, then, in order for \eqref{Wbound} to hold, the time $M_n$ would have to depend on the initial point $w$. This would destroy the uniform renormalization scheme used below.
\end{rem}

\begin{lem}\label{prop:eggb}
	For $k_n \leq i \leq M_n - \ell_n$, we have
  $$
    W_i\in\mathcal R_{\epsilon_i},
    \qquad
    \chi_{\epsilon_i}^{-1}(W_i)
    =
    \psi_{\epsilon_i}^{\iota}(w_i)
    \in\mathcal R_{\epsilon_i}.
$$
  % $W_i \in \mathcal{R}_{\eps_i}$ 
Moreover,
	$$W_{i} = W_{k_n} + \sum_{j=k_n}^{i-1}\eps_{j}+\tilde{A}(\eps_{j},w_j) $$
	and
$$\eps_{i}=\frac{1}{n+i+\phi_{p_0}^\iota(z_{0})}+O(\frac{1}{n^{1+\nu}}) $$
	%$$z_{i}=\frac{1}{i +(1-a_{3,0}+a_{2,1})\ln i +\phi_{p_0}^\iota(z_{0})+ \theta(w_0) -a_{2,1}\sum_{j=k_n}^{i-1}w_j+o(1)} $$
\end{lem}

\begin{proof}
	We prove the statement by induction $i$.
	\begin{itemize}
		\item 	Initialization: this follows from the fact that $W_{k_n} = \frac{k_n}{n}+O(\frac{k_n^2}{n^2})$ 
		(see Remark \ref{rem:initial}).
		\item Inductive step: Let $k_n \leq i\leq M_n - \ell_n$ and assume that $W_j\in\mathcal{R}_{\eps_j} $, $\psi^{\iota}_{\eps_j}(w_j)=(\chi_{\eps_j})^{-1}(W_j)\in\mathcal{R}_{\eps_j} $ and $\eps_j=\frac{1}{n+j+\phi_{p_0}^\iota(z_{0})}+o(1/n)$ for all $k_n \leq j<i$. 
    
    We need to prove that also $W_i\in\mathcal{R}_{\eps_i}$, $\psi^{\iota}_{\eps_i}(w_i)=(\chi_{\eps_i})^{-1}(W_i)\in\mathcal{R}_{\eps_i}$ and $\eps_i=\frac{1}{n+i+\phi_{p_0}^\iota(z_{0})}+o(1/n)$.
\medskip

First, recall that by Proposition \ref{prop:almosttransv2} we have
		
		\begin{align*}
			\left|W_{i}- W_{k_n} - \sum_{j=k_n}^{i-1}\eps_{j}\right|&=\left|\sum_{j=k_n}^{i-1}\tilde{A}(\eps_{j},w_j)\right|=O(\frac{1}{n})	
		\end{align*}
		where we have used the fact that $\tilde{A}(\eps_{j},w_j)= O(\eps_j^2)$ whenever $\psi^{\iota}_{\eps_j}(w_j)\in\mathcal{R}_{\eps_j}$ 
    %$\Phi_{\eps_j}(w_j)\in\mathcal{R}_{\eps_j}$ 
    and that $\eps_j=O(\frac{1}{n})$.

 It follows that
		\begin{align*}
			W_{i}&= W_{k_n} + \left(\sum_{j=k_n}^{i-1}z_{j}\right)+O(\frac{1}{n})\\
			&=\frac{k_n}{n}+\left(\sum_{j=k_n}^{i-1}\frac{1}{n+j}\right)+O(\frac{k_n^2}{n^2})\\
			&=\ln\left(\frac{n+i}{n+k_n}\right)+\frac{k_n}{n}+O(\frac{k_n^2}{n^2})\\
			&=\ln(1+\frac{i}{n})+O(\frac{k_n^2}{n^2})\\
		\end{align*}

	Next recall that by Lemma \ref{lem:invcomp} 
  %$$(\psi_{z}^{\iota/o})^{-1}(W)=-c z \cot(cW) - \frac{z}{2}+{O(z^2 \cot(cW),z^2)}$$ 
  $$\Phi_z^{-1}(W)=-c z \cot(cW) - \frac{z}{2}+{O(z^2 \cot(cW),z^2, z^{1+\delta}\log \frac{1}{|z|})}$$
  and 
$$
\phi_{p_0}^\iota(p(z,w))=\phi_{p_0}^\iota(z)+1-g(w)+O(zw).
$$
Observe that under these assumptions we have 
\begin{align*}
\phi_{p_0}^\iota(\eps_i )&=n+i+\phi_{p_0}^\iota(z_{0})-h(z_0,w_0)-\sum_{j=0}^{k_n-1}g(w_j)-\sum_{j=k_n}^{i-1}g(w_j)+o(1)\\
&=n+i+\phi_{p_0}^\iota(z_{0})-h(z_0,w_0)-\sum_{j=0}^{k_n-1}g(q_0^{j}(w_0))-a_{2,1}\sum_{j=k_n}^{i-1}w_j+o(1)
\end{align*}
and 
\begin{equation}\label{eq:rz}
  \sum_{j=k_n}^{i-1}w_j=\sum_{j=k_n}^{i-1}-c \eps_j \cot(cW_j) - \frac{\eps_j}{2}+{O(\eps_j^2 \cot(cW_j),\eps_j^2, \eps_j^{1+\delta}\log \frac{1}{\eps_j})},
\end{equation}

where
$$W_j=W_{k_n}+\sum_{s=k_n}^{j-1}\eps_{s} +O(\frac{1}{n})=\ln(1+\frac{j}{n})+O(k_n^2/n^2).$$

Using the Euler--Maclaurin formula and the fact that $|w_j|<|\eps_j|^\nu$ for all $k_n \leq j<i$, we obtain 

$$\sum_{j=k_n}^{i-1}w_j=-\ln(\frac{n}{k_n c} \sin(c\ln(1+\frac{i}{n}))) - \frac{1}{2}\ln(1+\frac{i}{n})+o(1),$$

%Furthermore since $i\leq M_n-\ell_n$ we can deduce that $z_i=\frac{1}{n+i+\phi_{p_0}^\iota(z_{0})}+o(1/n)$.

		Now observe that for $k_n\leq i\leq M_n-\ell_n$ 
$$
\ln(1+\frac{i}{n})\leq \frac{\pi}{c}-e^{-\frac{\pi}{c}}\frac{\ell_n}{n} =\frac{\pi}{c}-\frac{k_n}{n}+O(\frac{1}{n}),
$$
 and therefore $\eps_i=\frac{1}{n+i+\phi_{p_0}^\iota(z_{0})}+o(1/n)$ and

\begin{equation}\label{Wbound}
|\eps_{k_n}|^{1-\nu}+o(\eps_{k_n}^{1-\nu}) = \re(W_{k_n})<\re(W_{i})\leq \frac{\pi}{c}-\frac{k_n}{n}+O(\frac{k_n^2}{n^2})=\frac{\pi}{c}-|\eps_i|^{1-\nu}+o(\eps_i^{1-\nu})
\end{equation}
		and $ \im(W_{i}) = O(\frac{k_n^2}{n^2})$. All bounds in the $O(\cdot)$ terms above can be chosen to be independent of $i$. Therefore, there exists $N>0$ (independent of $i$), such that for every $n>N$ we will have $W_i\in\mathcal{R}_{\eps_i}$, provided that $W_j\in\mathcal{R}_{\eps_j}$ for all $k_n \leq j< i$. Finally, observe that for $W_i$ satisfying the condition \eqref{eq:rz} we also have 
    $$\psi^{\iota}_{\eps_i}(w_i)=\chi_{\eps_i}^{-1}(W_i)=W_i+O(R(\eps_i, W_i))=W_i+O(\eps_i\log\eps_i)=W_i+o(\eps_i^{1-\nu})\in\mathcal{R}_{\eps_i}.$$
	\end{itemize}	
\end{proof}

The preceding proof also shows that 
% $$
% \phi_{p_0}^\iota(\eps_{M_n-\ell_n} )=n+M_n-\ell_n+\phi_{p_0}^\iota(z_{0})-\sum_{j=0}^{k_n-1}g(q_0^{j}(w_0))+ a_{2,1}\frac{\pi}{2c}+o(1)$$

$$\eps_{M_n-\ell_n}=\left(n+M_n-\ell_n+(1-a_{3,0})\ln{e^{\frac{\pi}{c}}n}+\phi_{p_0}^\iota(z_{0})-h(z_0,w_0)-\sum_{j=0}^{k_n-1}g(q_0^{j}(w_0))+a_{2,1} \frac{\pi}{2c}+o(1)\right)^{-1}$$
and that $W_{M_n-\ell_n}= \frac{\pi}{c}-|\eps_{M_n-\ell_n}|^{1-\nu}+o(\eps_{M_n-\ell_n}^{1-\nu})$. We will need the following sharper estimate.

\begin{lem}\label{lem:sumeps}	We have 
	%W_{k_n}+\sum_{j=k_n}^{M_n-\ell_n-1} \eps_j + \Lambda \eps_j^2
	\begin{align*}
		W_{M_n-\ell_n}= \frac{\pi}{c}+\frac{G_n}{n}+ o\left(\frac{1}{n}\right)
	\end{align*}
	where
	$$G_n:=-e^{-\frac{\pi}{c}}\ell_n+(1- b_{0,3})e^{-\frac{\pi}{c}}\ln n+{\frac{k_n^2}{2n}}
+\phi^\iota_{{q_0}}(w_{0}) +\tau(z_0,w_0)\left(e^{-\frac{\pi}{c}}-1\right)+C_2-e^{-\frac{\pi}{c}}\rho_n,$$
	$\tau(z_0,w_0):=\phi_{p_0}^\iota(z_{0})-h(z_0,w_0)-\theta(w_0)$ and $C_2$ is an explicit constant. 
\end{lem}

\begin{proof}
	First, recall that by Proposition \ref{prop:almosttransv2} and Lemma \ref{prop:eggb} we have
	
	\begin{align*}
		W_{M_n-\ell_n} &= W_{k_n} + \sum_{j=k_n}^{M_n-\ell_n-1}\eps_{j}+\tilde{A}(\eps_{j},w_j)\\
		&= W_{k_n}+\sum_{j=k_n}^{M_n-\ell_n-1} \left(\eps_j + \Lambda \eps_j^2+O(\eps_j^{2+\delta})\right).
	\end{align*} 
	\color{black}
	Also recall that by Lemma \ref{lem:init} and Lemma \ref{lem:comp} we have:
	$$W_{k_n}=\frac{1}{n}\left(\phi^\iota_{q_0}(w_{0})+k_n+\frac{k_n^2}{2n} +(1- b_{0,3}) \ln n +E^\iota\right) + o\left(\frac{1}{n}\right)$$
	and that 
	$$
	\eps_j=\left(n+j+(1-a_{3,0}+\frac{a_{2,1}}{2})\ln(n+j)+a_{2,1}\ln(\frac{1}{c} \sin(c\ln(1+\frac{j}{n}))) + \frac{a_{2,1}}{2}\ln{n}+\tau(z_0,w_0)+o(1)\right)^{-1}.
	$$
where $\tau(z_0,w_0)=\phi_{p_0}^\iota(z_{0})-h(z_0,w_0)-\theta(w_0)$.
	\medskip

	First, observe that since $M_n=O(n)$ and $\eps_j=O(\frac{1}{j+n})\leq O(\frac{1}{n})$ we have
	$$ \sum_{j=k_n}^{M_n-\ell_n-1}\tilde{A}(\eps_{j},w_j)=\left(\sum_{j=k_n}^{M_n-\ell_n-1} \Lambda \eps_j^2\right)+O(\frac{1}{n^{1+\delta}})=\left(\sum_{j=k_n}^{M_n-\ell_n-1} \Lambda \eps_j^2\right)+o(\frac{1}{n})$$
	
	Next, define the functions $$\delta_1(j):=\frac{1}{n+j},$$
	 $$\delta_2(j):= -\frac{(1-a_{3,0}+\frac{a_{2,1}}{2})\ln(n+j)+\frac{a_{2,1}}{2}\ln{n}+\tau(z_0,w_0)}{(n+j)^2},$$
	$$\delta_3(j):= -\frac{a_{2,1}\ln(\frac{1}{ c} \sin(c\ln(1+\frac{j}{n})))}{(n+j)^2},$$

	 $$\delta_4(j):=\eps_j-\delta_1(j)-\delta_2(j)-\delta_3(j)=O\left(\frac{\ln^2 n}{n^3}\right).$$
	  Then 
	$\sum_{j=k_n}^{M_n-\ell_n-1} \delta_4(j) =O\left(\frac{\ln^2 n}{n^2}\right)=o\left(\frac{1}{n}\right),$ 
	and therefore
		$$\sum_{j=k_n}^{M_n-\ell_n-1} \eps_j  = \left( \sum_{j=k_n}^{M_n-\ell_n-1} \delta_1(j)+\delta_2(j) +\delta_3(j)\right) + o\left(\frac{1}{n}\right). $$ 
	
	Furthermore, by the Euler-MacLaurin formula applied to $\delta_1+\delta_2+\delta_3$, we get

		\begin{align*}
			\sum_{j=k_n}^{M_n-\ell_n-1} \delta_1(j)+\delta_2(j)+\delta_3(j) &=\int^{M_n-\ell_n}_{k_n} \delta_1(j)+\delta_2(j)+\delta_3(j) dj+ \frac{1}{2}(\delta_1(k_n)-\delta_1(M_n-\ell_n))\\ &\quad \quad \quad +\rho(\delta_1+\delta_2+\delta_3,k_n, M_n - \ell_n -1) \\
			&=\int^{M_n-\ell_n}_{k_n} \delta_1(j)+\delta_2(j)+\delta_3(j) dj+ \frac{1}{2n}(1-e^{-\pi/c})+o\left(\frac{1}{n}\right).
		\end{align*}

	We next compute the two integrals in the above expression:
	\begin{align*}
		\int^{M_n-\ell_n}_{k_n} \delta_1(j) dj&=\ln\left(\frac{n+M_n-\ell_n}{n+k_n}\right)\\
		&=\frac{\pi}{c}+( b_{0,3}-a_{3,0}+a_{2,1})(1-e^{-\frac{\pi}{c}})\frac{\ln n}{n}-e^{-\frac{\pi}{c}}\frac{1}{n}\rho_n-\frac{k_n}{n}-e^{-\frac{\pi}{c}}\frac{\ell_n}{n}\\
		&+\frac{1}{2n^2}(k^2_n-e^{-\frac{2\pi}{c}}\ell^2_n)+O\left(\frac{1}{n^{3(1-\nu)}}\right)\\
		&=\frac{\pi}{c}+( b_{0,3}-a_{3,0}+a_{2,1})(1-e^{-\frac{\pi}{c}})\frac{\ln n}{n}-e^{-\frac{\pi}{c}}\frac{1}{n}\rho_n-\frac{k_n}{n}-e^{-\frac{\pi}{c}}\frac{\ell_n}{n}+o\left(\frac{1}{n}\right)
	\end{align*}
	where we have used the fact that $\frac{1}{2}<\nu<\frac{2}{3}$ and that $k^2_n-e^{-\frac{2\pi}{c}}\ell^2_n=O(n^\nu)$. For the other integral, we have
	
	\begin{align*}
		\int^{M_n-\ell_n}_{k_n} \delta_2(j) dj&=((1-a_{3,0}+\frac{a_{2,1}}{2})+\tau(z_0,w_0)+\frac{a_{2,1}}{2}\ln{n})\left(\frac{1}{n+M_n-\ell_n}-\frac{1}{n+k_n}\right)\\
		& +(1-a_{3,0}+\frac{a_{2,1}}{2})\left(\frac{\ln(n+M_n-\ell_n)}{n+M_n-\ell_n}-\frac{\ln(n+k_n)}{n+k_n}\right)\\
		&=\frac{1}{n}(1-a_{3,0}+\frac{a_{2,1}}{2})\left(e^{-\frac{\pi}{c}}-1\right)+\frac{1}{n}\left(e^{-\frac{\pi}{c}}-1\right)\tau(z_0,w_0)+\frac{1}{n}(1-a_{3,0}+\frac{a_{2,1}}{2})e^{-\frac{\pi}{c}}\frac{\pi}{c}\\
		&+(1-a_{3,0}+\frac{a_{2,1}}{2})(e^{-\frac{\pi}{c}}-1)\frac{\ln{n}}{n}+\frac{a_{2,1}}{2}(e^{-\frac{\pi}{c}}-1)\frac{\ln{n}}{n}+o\left(\frac{1}{n}\right)\\
		&=\frac{1}{n}\left((1-a_{3,0}+\frac{a_{2,1}}{2}+\tau(z_0,w_0))\left(e^{-\frac{\pi}{c}}-1\right)+(1-a_{3,0}+\frac{a_{2,1}}{2})e^{-\frac{\pi}{c}}\frac{\pi}{c}\right)\\
		&+(1-a_{3,0}+a_{2,1})(e^{-\frac{\pi}{c}}-1)\frac{\ln{n}}{n}+o\left(\frac{1}{n}\right).
	\end{align*}

For the last integral, we use integration by parts to obtain:

\begin{align*}
		\int^{M_n-\ell_n}_{k_n} \delta_3(j) dj&=a_{2,1}\frac{\ln(c)}{n}(1-e^{-\frac{\pi}{c}})-a_{2,1}\int^{M_n-\ell_n}_{k_n}\frac{ \ln(\sin(c\ln(1+\frac{j}{n})))}{(n+j)^2}dj+o\left(\frac{1}{n}\right) \\
	&=a_{2,1}\frac{\ln(c)}{n}(1-e^{-\frac{\pi}{c}})-\frac{a_{2,1}}{n\cdot c}\int^{\pi-\frac{k_n c}{n}}_{\frac{k_n c}{n}}e^{-u/c}\ln(\sin(u))du+o\left(\frac{1}{n}\right)\\
&=a_{2,1}\frac{\ln(c)}{n}(1-e^{-\frac{\pi}{c}})-\frac{a_{2,1}}{n\cdot c}\int^{\pi}_{0}e^{-u/c}\ln(\sin(u))du+o\left(\frac{1}{n}\right)\\
%-a_{2,1}\int^{M_n-\ell_n}_{k_n} \frac{c\cot(c\ln{(1+\frac{j}{n})})}{(n+j)^2}dj+o\left(\frac{1}{n}\right)\\
%&=-\frac{a_{2,1}}{n}\int_{ck_n/n}^{\pi -ck_n/n+O(\ln{n}/n)}e^{-x/c}\cot(x)dx+o\left(\frac{1}{n}\right)
	\end{align*}

		Therefore, we have:
		\begin{align*}
			\sum_{j=k_n}^{M_n-\ell_n-1} \eps_j &= \frac{\pi}{c}+( b_{0,3}-1)(1-e^{-\frac{\pi}{c}})\frac{\ln n}{n}-e^{-\frac{\pi}{c}}\frac{1}{n}\rho_n-\frac{k_n}{n}-e^{-\frac{\pi}{c}}\frac{\ell_n}{n} \\
      &+\frac{\tau(z_0,w_0)}{n}\left(e^{-\frac{\pi}{c}}-1\right)+\frac{C_1}{n}+o\left(\frac{1}{n}\right).
		\end{align*}

		Next, observe that $\eps_j^2 =\delta_1^2(j)+O\left(\frac{\ln n}{n^3}\right)$,
		so that 
		$$\sum_{j=k_n}^{M_n-\ell_n-1} \eps_j^2 =\left(\sum_{j=k_n}^{M_n-\ell_n-1} \delta_1(j)^2\right) + O\left(\frac{\ln n}{n^2}\right). $$
		By the Euler--Maclaurin formula, we have:
		\begin{align*}
			\sum_{j=k_n}^{M_n-\ell_n-1} \delta_1(j)^2 
			&=\int^{M_n-\ell_n}_{k_n}\frac{1}{(n+j)^2}dj +O\left(\frac{1}{n^2}\right)\\
			&= \frac{1}{(n+k_n)}-\frac{1}{(n+M_n-\ell_n)} + o\left(\frac{1}{n}\right)\\
			&=\frac{1}{n}(1-e^{-\frac{\pi}{c}}) + o\left(\frac{1}{n}\right).
		\end{align*}	
		Therefore: 
		$$\sum_{j=k_n}^{M_n-\ell_n-1} \eps_j^2 = \frac{1}{n}(1-e^{-\frac{\pi}{c}}) + o\left(\frac{1}{n}\right).$$

	Putting everything together we obtain 		
	
	\begin{align*}
		W_{M_n-\ell_n}&= W_{k_n}+\left(\sum_{j=k_n}^{M_n-\ell_n-1} \eps_j + \Lambda \eps_j^2\right) +o(\frac{1}{n})\\
		&=
		\frac{\pi}{c}-e^{-\frac{\pi}{c}}\frac{\ell_n}{n}+(1- b_{0,3})e^{-\frac{\pi}{c}}\frac{\ln n}{n}+{\frac{k_n^2}{2n^2}}\\
&+\frac{1}{n}\left(\phi^\iota_{{q_0}}(w_{0}) +\tau(z_0,w_0)\left(e^{-\frac{\pi}{c}}-1\right)+C_2-e^{-\frac{\pi}{c}}\rho_n \right)+o\left(\frac{1}{n}\right),
	\end{align*}
where
	\begin{align*}
  C_2&=E^\iota+C_1+\Lambda(1-e^{-\frac{\pi}{c}})\\
 C_1&=(1-a_{3,0}+\frac{a_{2,1}}{2})(e^{-\frac{\pi}{c}}-1+\frac{\pi}{c}e^{-\frac{\pi}{c}})+(1-e^{-\frac{\pi}{c}})(a_{2,1}\ln c+\frac{1}{2})-a_{2,1}\int_{0}^\frac{\pi}{c}e^{-u}\log{ \sin(cu)}du.
 \end{align*}

\end{proof}

\begin{rem} Note that $G_n=-e^{-\frac{\pi}{c}}\ell_n+o(\ell_n)$ and therefore we have $\re(G_n)<0$ for all large $n$.
\end{rem} 

\color{black}

\subsection{Exiting the eggbeater}

For the proof of the following lemma, see \cite[Lemma 5.11]{ABT}.

\begin{lem}[Comparison with outgoing Fatou coordinates]\label{lem:exit}
	We have $w_{M_n-\ell_n} \in \mathcal{P}^o_R$, and
	$$\frac{1}{\eps_{M_n-\ell_n}}\left(\Phi_{\eps_{M_n-\ell_n}}(w_{M_n-\ell_n})-\frac{\pi}{c}\right)=\phi^o_{q_0}(w_{M_n-\ell_n})+{e^{-\frac{\pi}{c}}\frac{\ell_n^2}{2n}} +(1- b_{0,3}) \ln n + E^o + o(1)$$
	where 
	$E^o:=(1- b_{0,3}) \left(\frac{\pi}{c}-\ln c -e^{\frac{\pi}{c}} \int_{\pi/2c}^{\pi/c} e^{-u} \ln \sin(cu) du  \right)$.
\end{lem}

This lemma allows us to conclude the following:
		\begin{align*}
			\phi^o_{q_0}(w_{M_n-\ell_n}) &=\frac{1}{\eps_{M_n-\ell_n}}\left(\Phi_{\eps_{M_n-\ell_n}}(w_{M_n-\ell_n})-\frac{\pi}{c}\right)-e^{-\frac{\pi}{c}}\frac{\ell_n^2}{2n}-(1- b_{0,3}) \ln n - E^o + o(1)\\
			&=e^\frac{\pi}{c}\phi^\iota_{q_0}(w_{0})-\ell_n +(1-e^\frac{\pi}{c})\tau(z_0,w_0)-\rho_n+\Gamma +o(1)
		\end{align*}
		where the first equality follows from Lemma \ref{lem:exit} and the second equality follows from Lemma \ref{lem:sumeps} and Lemma \ref{prop:eggb}. In this computation we used the fact that $\frac{1}{2n}(e^\frac{\pi}{c}k_n^2- e^{-\frac{\pi}{c}}\ell_n^2)=o(1)$.
	
Next recall that 
$$
\phi_{p_0}^\iota(\eps_{M_n-\ell_n} )=n+M_n-\ell_n+\phi_{p_0}^\iota(z_{0})-h(z_0,w_0)-\sum_{j=0}^{k_n-1}g(q_0^{j}(w_0))+ a_{2,1}\frac{\pi}{2c}+o(1)
$$

and that, for $(\eps_j, w_j)\in \D(r,r) \times \D(r,r) \subset \mathcal{B}_{p_0}\times \C$ we have

\begin{equation*}
	\begin{array}{l}
		\phi_{p_0}^\iota(\eps_{j+1})=\phi_{p_0}^\iota(p_0(\eps_{j})+g(w_j)\eps_j^2+O(\eps^3_jw_j))=\phi_{p_0}^\iota(\eps_{j})+1-g(w_j)+O(\eps_jw_j) \\
  \phi_{q_0}^o(w_{j+1})=\phi_{q_0}^o(q_0(w_{j})+b\eps^2_j+O(\eps_j^3,\eps_jw_j^3,\eps_j^2w_j^2))=\phi_{q_0}^o(w_{j})+1+O(\frac{\eps_j^2}{w_j^2},\eps_jw_j)
	\end{array}
\end{equation*}

\begin{lem}\label{lem:compout}
	 For $M_n - \ell_n \leq j \leq M_n$, we have $w_j \in \mathcal{P}_{R}^o$ and 
	$$\phi^o_{q_0}(w_{j}) = \phi^o_{q_0}(w_{M_n-\ell_n}) + j -(M_n-\ell_n)+ o(1).$$
\end{lem}
\begin{proof} Recall that $w_{M_n-\ell_n}= \frac{1}{\ell_n} + o\left(\frac{1}{n}\right).$ 
	As in Lemma \ref{lem:init}, we prove inductively that
	 $$\phi^o_{q_0}(w_{j+1})=\phi^o_{q_0}(w_j)+1+ O\left(\frac{\ell_n^2}{n^2} \right)$$ 
	 and
\begin{equation}\label{z-estimate}
\phi_{p_0}^\iota(\eps_{j} )=n+j+\phi_{p_0}^\iota(z_{0})-h(z_0,w_0)-\sum_{i=0}^{k_n-1}g(q_0^{i}(w_0))+ a_{2,1}\frac{\pi}{2c}-\sum_{i=M_n-\ell_n}^{j-1}g(w_i)+o(1)
\end{equation}
	 for all $M_n-\ell_n\leq j\leq M_n-1$. In particular $\phi^o_{q_0}(w_{M_n})=\phi^o_{q_0}(w_{M_n-\ell_n})+\ell_n + o(1)$, where we use the fact that $O\left(\frac{\ell_n^3}{n^2}\right)=o(1)$ since $\ell_n\sim n^{\nu}$ for $\nu\in\left(\frac{1}{2},\frac{2}{3}\right)$. 
\end{proof}

\subsection{Conclusion}

	\begin{proof}[Proof of Theorem \ref{th:maintech}] From the previous subsection, we conclude that 
		$$\phi^o_{q_0}(w_{M_n})= e^{\frac{\pi}{c}}\phi^\iota_{q_0}(w_{0})-\left(e^{\frac{\pi}{c}}-1\right)\phi^{\iota}_p(z_0)+(e^\frac{\pi}{c}-1)\theta(w_0)+\left(e^{\frac{\pi}{c}}-1\right)h(z_0,w_0)-\rho_n +\Gamma+ o(1).$$
		 where
    \begin{multline}\label{eq:Gamma}
			\Gamma:=(e^{\frac{\pi}{c}}-1)\left(\frac{a_{3,0}- b_{0,3}+b_{3,0}-a_{2,1}}{2b}+a_{3,0}-a_{2,1}+\frac{1}{2}(1- b_{0,3})\right)\\
      +(e^{\frac{\pi}{c}}-1)( b_{0,3}-1+a_{2,1})\ln{c} 	+\frac{\pi}{c}( b_{0,3}-a_{3,0}+\frac{a_{2,1}}{2})\\	
      +	e^{\frac{\pi}{c}}(1- b_{0,3}-a_{2,1}) \int_0^{\frac{\pi}{c}} e^{-u} \ln \sin(cu) du.
		\end{multline}
    is obtained from the equation $\Gamma=\alpha_0C_2-E^o$.
    In particular we have 
		$$w_{M_n}=\mathcal{L}\left(e^{\frac{\pi}{c}},\Gamma-\rho_n+\left(e^{\frac{\pi}{c}}-1\right)h(z_0,w_0);z_0,w_0\right)+ o(1)$$
	\end{proof}

\section{Applications of the main Theorem}

In this section we will prove Corollaries \ref{top} and \ref{prop:wd}. Let $P$ be a map of the form \eqref{map1} and throughout this section assume that $g(w)=0$, i.e. $p(z,w)$ contains no $z^2w^{j}$ terms (for $j\geq 1$). See Remark \ref{rem:g-nonzero-obstruction} for an explanation of why this assumption is essential for applications. Furthermore let $\alpha_0,\beta_0$ and $b>\frac{1}{4}$ be as in Theorem \ref{th:maintech} and assume that there exists an $(\alpha_0,\beta_0)$-admissible sequence $(n_k)_{k \geq 0}$ with a convergent phase sequence $(\sigma_k)_{k \geq 0}$ whose limit we denote by $\sigma$.

By definition of $M_n$ and $\rho_n$, we have
		$$M_{n_k}=\lfloor (\alpha_0-1) n_k + \beta_0 \ln n_k\rfloor,$$ and 
		$\rho_{n_k}=\{(\alpha_0-1) n_k + \beta_0 \ln n_k \}$. Therefore, by the definition of an $(\alpha_0,\beta_0)$-admissible sequence,
		there exists a bounded sequence of integers $(m_k)_{k \geq 0}$ such that 
		$$n_{k+1}-n_k = M_{n_k}+m_k,$$
		and the phase sequence of $(n_k)_{k \geq 0}$ is given by
		\begin{align*}
			\sigma_{k}= n_{k+1}-\alpha_0 n_k-\beta_0 \ln n_k &= n_{k+1}- (M_{n_k}+n_k+\rho_{n_k})\\
			&=m_k-\rho_{n_k}.
		\end{align*}
Since $\sigma_k$ converges there exists a non-negative integer $m$ such that $m_k-m\in\{0,1\}$ for all $k\geq0$. Using the notation from the previous section, observe that for all $0\leq j\leq m_k$ we have 
\begin{equation*}
	\begin{array}{l}
		\phi_{p_0}^\iota(\eps_{{M_{n_k}}+j+1})=\phi_{p_0}^\iota(\eps_{{M_{n_k}}+j})+1+O(\eps_{M_{n_k}+j}w_{M_{n_k}+j}) \\
  \phi_{q_0}^o(w_{{M_{n_k}}+j+1})=\phi_{q_0}^o(w_{{M_{n_k}}+j})+1+O\left(\frac{\eps_{{M_{n_k}}+j}^2}{w_{{M_{n_k}}+j}^2},\eps_{{M_{n_k}}+j}w_{{M_{n_k}}+j}\right)
	\end{array}
\end{equation*}
and therefore, by induction, 

\begin{align*}
		\phi_{p_0}^\iota(\eps_{{M_{n_k}}+m_k})&=\phi_{p_0}^\iota(p_0^{n_{k+1}}(z_{0}))-h(z_0,w_0)+o(1) \\
  \phi_{q_0}^o(w_{{M_{n_k}}+m_k})&=
  \phi_{q_0}^o(w_{M_{n_k}})+m_k+o(1)= \phi_{q_0}^o(q_0^{m_k}(w_{M_{n_k}}))+o(1)\\
  &= \phi_{q_0}^o\left(\mathcal{L}\left(e^{\frac{\pi}{c}},\Gamma+m_k-\rho_n+\left(e^{\frac{\pi}{c}}-1\right)h(z_0,w_0);z_0,w_0\right)\right)+o(1).
\end{align*}
It follows that 
\begin{equation}\label{eq:special}
	P^{n_{k+1}-n_k}\left(p_0^{n_k}(z)+\frac{h(z,w)}{n_k^2}+o(n_k^{-2}),w\right) = \left(0, \lcal(\alpha_0, \Gamma + \sigma; z,w) \right) + o(1) 
	\end{equation}

To simplify the notation we write  $ \lcal_z(w):= \lcal(\alpha_0, \Gamma + \sigma; z,w)$. By \cite[Lemma 6.1]{ABT} we know that there exists $(z_0,w_0)\in \bcal_{p_0} \times \bcal_{q_0}$ such that $w_0$ is a super-attracting fixed point of $ \lcal_{z_0}(w)$. Moreover $q_0^k(w_0)$ is a super-attracting fixed point of $ \lcal_{p_0^{k}(z_0)}(w)$ for all $k\geq 0$. 
This implies that we can find $r>0$ such that, for every $z\in \D(z_0,r)$, the map $\lcal_z$ has an attracting fixed point in $\D(w_0,r)$ and these fixed points vary holomorphically with $z$. Thus there is a non-constant holomorphic function $\zeta$ such that $\zeta(z)$ is an attracting fixed point of $\lcal_z$ and $\zeta(z_0)=w_0$. Consequently, $\lcal_z(\D(\zeta(z),r))\Subset \D(\zeta(z),r)$ for all $z\in \D(z_0,r)$. By taking $r>0$ even smaller, we can make sure that the same holds for any sufficiently small holomorphic perturbation of the map $\lcal_z(w)$.

\begin{prop}\label{prop:main} Let $P$ be as in Theorem \ref{th:maintech} with $g(w)=0$ and assume that there exists an $(\alpha_0,\beta_0)$-admissible sequence with a convergent phase sequence $\sigma_k\rightarrow \sigma$. Let $w_0$ be a super-attracting fixed point of $ \lcal_{z_0}(w)$ and set $V_j:=P^{-n_j}\left(p^{n_j}_0(\D(z_0,r))\times \D(w_0,r)\right)$ for some sufficiently small $r$. For all sufficiently large $j$, the set $V_j$ is contained in a Fatou component $\Omega$ of $P$. Moreover, some subsequence of $P^{n_k}|_{\Omega}$ converges to a rank one limit map.

\end{prop}

	\begin{proof} Define

$$\eps_{k}:=\proj_1(P^{n_{k+j}-n_j}(p_0^{n_j}(z),w)),$$
and 
$$w_{k}:=\proj_2(P^{n_{k+j}-n_j}(p_0^{n_j}(z),w)).$$

We claim that for all sufficiently large $j$ we have
$$\phi^{\iota}_p(\eps_{k}):=n_{k+j}+\phi_{p_0}^\iota(z)+ h(z,w)+o(1) \qquad (\text{ as } k \to +\infty) $$
 and that the sequence
$$\{w_{k}\}_{k\geq 0}\subset \D(w_0,r)$$
 is convergent.
 \medskip

From the proof of Theorem \ref{th:maintech}, we deduce 
$$\phi^{\iota}_p(\eps_{1})=n_{j+1}+\phi_{p_0}^\iota(z)+ h_{j}(z,w)$$
where $h_{j}(z,w)=O\left(\frac{1}{n_j^{\nu}} \right)$ with uniform bounds on $\D(z_0,r)\times \D(w_0,r)$ for some $\nu>0$. By the initial assumption, 
$w_{1}=\lcal_z(w)+o(1)$ and therefore $|w_{1}-\lcal_z(w)|<\lambda|w-\zeta(z)|$ for some uniform $0<\lambda<1$.

In the next step, we obtain 

$$\phi^{\iota}_p(\eps_{2})=n_{j+2}+\phi_{p_0}^\iota(z)+ h_{j}(z,w)+h_{j+1}(z,w)$$
where $h_{j+1}(z,w)=O\left(\frac{1}{n_{j+1}^{\nu}} \right)$ with uniform bounds on $\D(z_0,r)\times \D(w_0,r)$ for some $\nu>0$.
Moreover $w_{2}=\lcal_{z}^1\circ \lcal^0_{z}(w)+o(1)$ where $\lcal_{z}^0:=\mathcal{L}_z$ and $\lcal_{z}^1(w)=\mathcal{L}(e^{\frac{\pi}{c}},\Gamma+\sigma+(\alpha_0-1)h_{j}(z,w);z,w)$, hence $w_{j+2,j}\in\D(w_0,r)$ if we have started with $j$ sufficiently large.

By induction we end up with 
$$\phi^{\iota}_p(\eps_{k})=n_{j+k}+\phi_{p_0}^\iota(z)+ \sum_{\ell=0}^{k-1}h_{j+\ell}(z,w)$$
where all $h_{\ell}(z,w)=O\left(\frac{1}{n_{\ell}^{\nu}} \right)$ have uniform bounds on $\D(z_0,r)\times \D(w_0,r)$.

Moreover we have $w_{k}=\lcal_{z}^{k-1}\circ\ldots \circ \lcal_{z}^0(w)+o(1)$ as $k \to +\infty$ where 
$$\lcal_{z}^i(w):=\mathcal{L}(e^{\frac{\pi}{c}},\Gamma+\sigma+(\alpha_0-1)\sum_{\ell=0}^{i-1}h_{j+\ell}(z,w);z,w).$$

Since $h_{\ell}(z,w)=O\left(\frac{1}{n_{\ell}^{\nu}} \right)$ and $n_{\ell}=O(\alpha_0^{\ell})$ (see \cite[Lemma 10.1]{ABT}) we have 
$$\sum_{\ell=0}^{k-1}h_{\ell+j}(z,w)=O \left(\sum_{\ell=0}^{k-1}\frac{1}{(\alpha_0^{j+\ell})^\nu}\right)=O\left(\frac{1}{n_{j}^{\nu}} \right),$$
hence $\lcal_{z}^k(w)=\lcal_z(w)+O\left(\frac{1}{n_{j}^{\nu}} \right)$ with uniform bounds on $\D(z_0,r)\times \D(w_0,r)$ with respect to $k$. This implies that $w_{k}\in\D(w_0,r)$.
\medskip
 
We can also conclude that there exists a holomorphic function $h$ on $\D(z_0,r)\times \D(w_0,r)$ such that $\lcal_{z}^k(w)=\mathcal{L}(e^{\frac{\pi}{c}},\Gamma+\sigma+(\alpha_0-1)h(z,w);z,w)+o(1)$.  In particular it follows that for all sufficiently large $j$, there is a non-constant holomorphic function $\tilde{\zeta}:\D(z_0,r)\rightarrow \D(w_0,r)$ with $|\tilde{\zeta}(z)-\zeta(z)|\lesssim \frac{1}{n_{j}^{\nu}}$ such that
\begin{equation}\label{eq:lim}
  \lim_{k\rightarrow\infty} P^{n_k-n_j}(p_0^{n_j}(z),w)=(0,\tilde\zeta(z))
\end{equation}
uniformly on $\D(z_0,r)\times \D(w_0,r)$.

\medskip
Normality of the family of iterates of $P$ on $V_j$ follows from the proof of Theorem \ref{th:maintech}, which gives uniform estimates along the orbit of the compact set between the times $n_k$ and $n_{k+1}$ for each $k\geq 0$. In particular, if $\mathbb{B}(0,R)$ is a Euclidean ball centered at the origin and containing $p_0^{n_j}(\D(z_0,r))\times \D(w_0,r)$ then 
$P^{n_j+k}(V_j) \subset \mathbb{B}(0,R)$ for all $k\geq 0$. Therefore, there is a Fatou component $\Omega$ of $P$ that contains $V_j$. From \eqref{eq:lim} and the identity principle, we conclude that some subsequence of $P^{n_k}|_{\Omega}$ converges to a rank one limit map.

\end{proof}

 We first briefly describe the idea of the proof of Corollary \ref{top}. We have seen that for a fixed point of the Lavaurs map $\lcal_{z_0}(w)$, the time needed to pass through the eggbeater (gates) and to return approximately to its initial position grows asymptotically like $n_{k+1}^1-n_k^1\sim\alpha_1^k$. On the other hand, if $P_2$ is conjugate to $P_1$, one should expect that this return time stays the same, but we know that return times for $P_2$ grow asymptotically like $n_{k+1}^2-n_k^2\sim\alpha_2^k$ which leads to a contradiction if $\alpha_1\neq \alpha_2$. A similar argument is then used for $\beta_i$'s.

\subsection{Proof of Corollary \ref{top}}

To simplify notation, write
$$
    p_i(z):=\proj_1(P_i(z,0)),
    \qquad
    q_i(w):=\proj_2(P_i(0,w)),
    \qquad i=1,2.
$$
These are the one-variable parabolic germs denoted previously by $p_0$ and
$q_0$ for the corresponding map $P_i$.

Let $\mathfrak h:U\to V$
be a homeomorphism defined in neighborhoods of the origin such that
$$
    \mathfrak h\circ P_1=P_2\circ\mathfrak h,
    \qquad
    \mathfrak h(0,0)=(0,0).
$$
We prove that
$$
    (\alpha_1,\beta_1)=(\alpha_2,\beta_2).
$$

Let $(n_k^1)$ be an $(\alpha_1,\beta_1)$-admissible sequence whose phase
sequence converges. Choose
$$
    (z_0,w_0)\in \mathcal B_{p_1}\times\mathcal B_{q_1}
$$
such that $w_0$ is a super-attracting fixed point of the Lavaurs map
corresponding to $P_1$. Replacing $(z_0,w_0)$ by
$$
    (p_1^\ell(z_0),q_1^\ell(w_0))
$$
for some sufficiently large $\ell$, we may assume that $(z_0,w_0)\in U$.

By Proposition \ref{prop:main}, there exist $r>0$ and $j\in\mathbb N$ such
that, for every $k\ge0$,
$$
    P_1^k\left(p_1^{n_j^1}(z),w\right)\in U
    \qquad
    \text{for all }(z,w)\in\D(z_0,r)\times\D(w_0,r).
$$
Therefore
\begin{equation}\label{eq:conjn-corrected}
    \mathfrak h\circ
    P_1^k\left(p_1^{n_j^1}(z),w\right)
    =
    P_2^k\circ
    \mathfrak h\left(p_1^{n_j^1}(z),w\right)
    \in V
\end{equation}
for all $(z,w)\in\D(z_0,r)\times\D(w_0,r)$ and all $k\ge0$.

Moreover, by the construction in Proposition \ref{prop:main}, we have
$$
    P_1^{n_k^1-n_j^1}
    \left(p_1^{n_j^1}(z_0),w_0\right)
    \longrightarrow
    (0,\widetilde\zeta(z_0)),
$$
where $\widetilde\zeta(z_0)\neq0$. Set
$$
    (x_0,y_0):=
    \mathfrak h\left(p_1^{n_j^1}(z_0),w_0\right),
    \qquad
    N_k:=n_k^1-n_j^1.
$$
Then, by \eqref{eq:conjn-corrected},
\begin{equation}\label{eq:convp-corrected}
    P_2^{N_k}(x_0,y_0)
    \longrightarrow
    \mathfrak h(0,\widetilde\zeta(z_0)).
\end{equation}

We first show that this limit belongs to the invariant line $\{x=0\}$.
Write
$$
    P_2^m(x_0,y_0)=(x_m,y_m).
$$
By \eqref{eq:conjn-corrected}, the orbit $(x_m,y_m)$ remains in a sufficiently
small neighborhood of the origin. In particular, $y_m$ is bounded and small
for all $m\ge0$. Since $g\equiv0$, the first coordinate of $P_2$ has the
form
$$
    \proj_1 P_2(x,y)=p_2(x)+O(x^3y).
$$
Thus
$$
    x_{m+1}=p_2(x_m)+O(x_m^3y_m).
$$
Applying the incoming Fatou coordinate of $p_2$, we obtain
$$
    \phi_{p_2}^{\iota}(x_{m+1})
    =
    \phi_{p_2}^{\iota}(x_m)+1+O(x_my_m).
$$
Since $y_m=O(1)$ and the orbit remains in the attracting neighborhood, we have
$x_m=O(1/m)$. Hence
$$
    \phi_{p_2}^{\iota}(x_m)
    =
    \phi_{p_2}^{\iota}(x_0)+m+O(\log m).
$$
Equivalently,
$$
    x_m=p_2^m(x_0)+O\left(\frac{\log m}{m^2}\right).
$$
In particular $x_m\to0$. Applying this to the subsequence $m=N_k$, we get
$$
    \proj_1\mathfrak h(0,\widetilde\zeta(z_0))
    =
    \lim_{k\to\infty}x_{N_k}
    =
    0.
$$
Therefore
$$
    \mathfrak h(0,\widetilde\zeta(z_0))=(0,w_3)
$$
for some $w_3\in\mathbb C$.

We next show that $w_3\in\mathcal B_{q_2}\setminus\{0\}$. Since
$$
    \mathfrak h\circ P_1^\ell=P_2^\ell\circ\mathfrak h
$$
and since
$$
    P_1^\ell(0,\widetilde\zeta(z_0))=(0,q_1^\ell(\widetilde\zeta(z_0))),
$$
we have
$$
    P_2^\ell(0,w_3)
    =
    (0,q_2^\ell(w_3))
$$
for every $\ell\ge0$. Moreover this orbit remains in a sufficiently small
neighborhood of the origin. By the one-dimensional Leau--Fatou theory for the
parabolic germ $q_2$, it follows that either $w_3=0$ or
$w_3\in\mathcal B_{q_2}$. Since $\widetilde\zeta(z_0)\neq0$,
$\mathfrak h(0,0)=(0,0)$, and $\mathfrak h$ is injective, we cannot have
$w_3=0$. Hence
$$
    w_3\in\mathcal B_{q_2}\setminus\{0\}.
$$

We have therefore obtained a genuine return sequence for the $P_2$-orbit:
\begin{equation}\label{eq:P2-return}
    P_2^{N_k}(x_0,y_0)=(x_{N_k},y_{N_k})
    \longrightarrow
    (0,w_3),
    \qquad
    w_3\in\mathcal B_{q_2}\setminus\{0\}.
\end{equation}
Moreover,
\begin{equation}\label{eq:roughphase}
    \phi_{p_2}^{\iota}(x_{N_k})
    =
    N_k+O(\log N_k).
\end{equation}

We now prove that $\alpha_1=\alpha_2$. Suppose, by contradiction, that
$\alpha_1<\alpha_2.$ Let
$$
    m_k:=N_{k+1}-N_k=n_{k+1}^1-n_k^1.
$$
Since $(n_k^1)$ is $(\alpha_1,\beta_1)$-admissible, we have
$$
    m_k=(\alpha_1-1)N_k+\beta_1\log N_k+O(1)
    =
    (\alpha_1-1)N_k+O(\log N_k).
$$
Hence, since $\alpha_1<\alpha_2$, there exists $\eta>0$ such that, for all
large $k$,
$$
    k_{N_k}=\lfloor N_k^{\nu}\rfloor<m_k<(\alpha_2-1-\eta)N_k,
$$
where $\nu\in(\frac{1}{2},\frac{2}{3})$.
The second coordinate $y_{N_k}$ belongs to a compact subset $K$ of
$\mathcal B_{q_2}$, since $y_{N_k}\to w_3$. Together with \eqref{eq:roughphase}, this allows us to apply the robust form of the eggbeater
estimate to the orbit segment of $P_2$ starting from  $(x_{N_k},y_{N_k}).$
By the same argument as in the proof of Lemma \ref{prop:eggb} we obtain 
$$
    y_{N_{k+1}}=\prw P_2^{m_k}(x_{N_k},y_{N_k})=O(N_k^{-\nu})=o(1).
$$

This contradicts \eqref{eq:P2-return}, since $y_{N_{k+1}}\to w_3\neq0$.
Therefore $\alpha_1\ge\alpha_2$ and hence by symmetry, $\alpha_1=\alpha_2$.

It remains to prove that $\beta_1=\beta_2$. We write
$$
    \alpha:=\alpha_1=\alpha_2.
$$

Recall that we have constructed a point $(x_0,y_0)$ and a sequence
$$
    N_k:=n_k^1-n_j^1
$$
such that
$$
    P_2^{N_k}(x_0,y_0)=(x_{N_k},y_{N_k})
    \longrightarrow
    (0,w_3),
    \qquad
    w_3\in\mathcal B_{q_2}\setminus\{0\}.
$$
Moreover, the orbit remains in a sufficiently small neighborhood of the origin.

We shall use the following consequence of the estimates from the proof of
Theorem \ref{th:maintech}. Along the above return sequence, the first coordinate
has a bounded phase correction: there exist $\xi\in\mathcal B_{p_2}$ and a
bounded sequence $(h_k)$ such that
\begin{equation}\label{eq:bounded-phase-beta}
    x_{N_k}
    =
    p_2^{N_k}(\xi)+\frac{h_k}{N_k^2}+o(N_k^{-2}).
\end{equation}
Equivalently,
\begin{equation}\label{eq:bounded-fatou-phase-beta}
    \phi_{p_2}^{\iota}(x_{N_k})
    =
    N_k+C+o(1)
\end{equation}
for some constant $C$.

Let us briefly explain why \eqref{eq:bounded-fatou-phase-beta} follows from the
previous estimates. Since $g\equiv0$, the first coordinate of $P_2$ satisfies
$$
    \proj_1P_2(x,y)=p_2(x)+O(x^3y).
$$
Thus, along the orbit,
$$
    \phi_{p_2}^{\iota}(x_{m+1})
    =
    \phi_{p_2}^{\iota}(x_m)+1+O(x_my_m).
$$
After the equality $\alpha_1=\alpha_2$ is known, the intervals
$$
    [N_k,N_{k+1}]
$$
have the correct eggbeater length for $P_2$, up to an $O(\log N_k)$-error.
Applying the fact that $x_m=p_2^m(x_0)+O\left(\frac{\log m}{m^2}\right)$ for all $m\geq 0$ to the proofs of Lemma \ref{lem:init}, Lemma \ref{prop:eggb} and Lemma \ref{lem:compout} the entry, eggbeater, and exit estimates imply
$$
    \sum_{m=N_k}^{k_{N_{k}}-1} x_my_m
    =
    O\left(\frac{\log N_k}{N_k}\right), \quad \sum_{m=k_{N_k}}^{N_{k}-\ell_{N_{k}}-1} x_my_m
    =
    O\left(\frac{\log N_k}{N_k}\right),\quad
\sum_{m=N_{k}-\ell_{N_{k}}}^{N_{k+1}-1} x_my_m
    =
    O\left(\frac{\log N_k}{N_k}\right).
    $$
Consequently, if
$$
    C_k:=\phi_{p_2}^{\iota}(x_{N_k})-N_k,
$$
then
$$
    C_{k+1}-C_k
    =
    O\left(\frac{\log N_k}{N_k}\right).
$$
Since $N_k\sim \alpha^k$ grows exponentially, the series
$$
    \sum_k\frac{\log N_k}{N_k}
$$
converges. Hence $(C_k)$ is a Cauchy sequence, and in particular $C_k=C+o(1)$.
This proves \eqref{eq:bounded-fatou-phase-beta}, and hence
\eqref{eq:bounded-phase-beta}.

We now argue by contradiction. Suppose that
$$
    \beta_1<\beta_2.
$$
Set
$$
    m_k:=N_{k+1}-N_k=n_{k+1}^1-n_k^1.
$$
Since $(n_k^1)$ is $(\alpha,\beta_1)$-admissible, and since
$N_k=n_k^1-n_j^1$, we have
\begin{equation}\label{eq:mk-beta1}
    m_k
    =
    (\alpha-1)N_k+\beta_1\log N_k+O(1).
\end{equation}
On the other hand, the natural exit time for $P_2$, with parameter $N_k$, is
$$
    M_{N_k}^{(2)}
    :=
    \left\lfloor
    (\alpha-1)N_k+\beta_2\log N_k
    \right\rfloor.
$$
Therefore, by \eqref{eq:mk-beta1},
$$
    m_k
    =
    M_{N_k}^{(2)}
    -
    (\beta_2-\beta_1)\log N_k
    +
    O(1).
$$
Since
$$
    \log N_k=o(\ell_{N_k}),
$$
it follows that, for all sufficiently large $k$,
$$
    M_{N_k}^{(2)}-\ell_{N_k}
    <
    m_k
    <
    M_{N_k}^{(2)}.
$$

Now apply Theorem \ref{th:maintech}, Lemma \ref{lem:exit}, and
Lemma \ref{lem:compout} to the orbit segment of $P_2$ starting at    $(x_{N_k},y_{N_k})$.
This is justified by \eqref{eq:bounded-phase-beta}, since the bounded sequence
$h_k$ in \eqref{eq:bounded-phase-beta} only changes the Lavaurs phase by a
bounded amount. In particular, because $m_k$ is
$$
    (\beta_2-\beta_1)\log N_k+O(1)
$$
steps before the natural exit time $M_{N_k}^{(2)}$, the outgoing Fatou
coordinate satisfies
$$
    \phi_{q_2}^{o}(y_{N_{k+1}})
    =
    -(\beta_2-\beta_1)\log N_k+O(1).
$$
Hence
$$
    y_{N_{k+1}}
    \sim
    \frac{1}{(\beta_2-\beta_1)\log N_k}
    \longrightarrow 0.
$$
This contradicts the fact that
$$
    y_{N_{k+1}}\longrightarrow w_3\neq0.
$$
Therefore $\beta_1\geq\beta_2$.

By applying the same argument with the roles of $P_1$ and $P_2$ interchanged,
we obtain the reverse inequality $ \beta_2\geq\beta_1$ and thus $\beta_1=\beta_2$.

\subsection{Proof of Corollary \ref{prop:wd}}
% The conclusion of the corollary follows directly from Proposition \ref{prop:main} and the fact that by Lemma \ref{lem:init} we have $P^{n_k+ \lfloor n_k^\nu\rfloor}|_{V_j}\sim O(\frac{1}{\lfloor n_k^\nu\rfloor})=o(1)$.

By Proposition \ref{prop:main}, for all sufficiently large $j$, the set
$V_j$ is contained in a Fatou component $\Omega$, and a subsequence
$P^{n_k}|_{\Omega}$ converges to a rank-one limit map.

On the other hand, Lemma \ref{lem:init}, applied uniformly on
$p_0^{n_j}(\D(z_0,r))\times \D(w_0,r)$, gives
$$
    P^{n_k+\lfloor n_k^\nu\rfloor}|_{V_j}=O(n_k^{-\nu})\to(0,0)
$$
locally uniformly on $V_j$. Since the family of iterates is normal on
$\Omega$, the same subsequential limit extends to $\Omega$. Thus
$\Omega$ also admits the constant limit map $(0,0)$, which has rank zero.
Hence $\Omega$ admits both rank-one and rank-zero limit maps.

 \begin{rem}
 We expect that the Fatou component $\Omega$ constructed above is wandering, but we have not been able to prove this. Note that there is a longstanding open question about the existence of an invariant Fatou component admitting both a rank-one and a rank-zero limit map, for endomorphisms of $\C^n$. The existence of such domains is known for non-globally defined maps, see \cite[Theorem 31]{LP}.

% If $\Omega$ were non-wandering, then after replacing $P$ by a suitable iterate, this would give an invariant Fatou component admitting both a rank-one and a rank-zero limit map. Such an example would answer a longstanding open question the existence of such domains. 
 \end{rem}

\begin{rem}\label{rem:g-nonzero-obstruction}
The assumption $g\equiv0$ in the applications above is essential. We briefly
explain the obstruction that appears in the general case.

Let $(n_k)$ be an $(\alpha_0,\beta_0)$-admissible sequence with converging
phase sequence. Let $k$ be sufficiently large, and suppose that, for some
$$
    (z_0,w_0)\in K'\times K\Subset
    \mathcal B_{p_0}\times \mathcal B_{q_0},
$$
the first few renormalized returns remain in $K$. More precisely, assume that
for $0\le j\le \ell$,
$$
    w_j
    :=
    \operatorname{proj}_2
    P^{n_{k+j}-n_k}\bigl(p_0^{n_k}(z_0),w_0\bigr)
    \in K.
$$

Recall that the first coordinate satisfies
$$
    \phi_{p_0}^{\iota}(\epsilon_{m+1})
    =
    \phi_{p_0}^{\iota}(\epsilon_m)
    +1-g(w_m)+O(\epsilon_m w_m).
$$
Hence, during one complete passage through the eggbeater, the first-coordinate
Fatou phase receives a contribution from the sum of $g$ along the incoming and
outgoing parts of the $w$-orbit.

The incoming contribution is encoded by 
$$
    \theta(w)
    :=
    \lim_{N\to\infty}
    \left(
    \sum_{j=0}^{N-1} g(q_0^j(w))
    +
    a_{2,1}\log N
    \right).
$$
There is also an outgoing contribution. Namely, on the outgoing petal, let
$q_0^{-1}$ denote the inverse branch of $q_0$, and define
$$
    \widetilde\theta(w)
    :=
    \lim_{N\to\infty}
    \left(
    \sum_{j=0}^{N-1} g(q_0^{-j}(w))
    -
    a_{2,1}\log N
    \right).
$$
With the normalizations used above, the correction to the first-coordinate phase
after one passage is of the form
$$
    \mathcal E_\sigma(z,w)
    :=
    \theta(w)
    +
    \widetilde\theta\left(
    \mathcal L(\alpha_0,\Gamma+\sigma;z,w)
    \right)
    +
    a_{2,1}\frac{\pi}{2c}.
$$
%Here $\sigma$ denotes the phase limit of the admissible sequence.
Indeed, from the estimate \eqref{z-estimate} we can conclude that
$$
\begin{aligned}
    z_1
    &:=
    \operatorname{proj}_1
    P^{n_{k+1}-n_k}
    \bigl(p_0^{n_k}(z_0),w_0\bigr)                   \\
    &=
    p_0^{n_{k+1}}(z_0)
    +
    \frac{\mathcal E(z_0,w_0)}{n_{k+1}^2}
    +
    o\left(\frac1{n_{k+1}^2}\right).
\end{aligned}
$$
Therefore, by induction, provided that the renormalized $w$-orbit remains in the
compact set $K$, one obtains
$$
    z_j
    =
    p_0^{n_{k+j}}(z_0)
    +
    \frac{
    \sum_{i=0}^{j-1}\mathcal E(z_0,w_i)}
    {n_{k+j}^2}
    +
    o\left(\frac1{n_{k+j}^2}\right),
    \qquad 1\le j\le \ell,
$$
where the $w$-coordinate satisfies
$$
    w_j
    =
    \mathcal L\left(
    \alpha_0,
    \Gamma+\sigma+
    (\alpha_0-1)
    \sum_{i=0}^{j-2}\mathcal E(z_0,w_i);
    z_0,w_{j-1}
    \right)
    +
    o(1),
    \qquad 1\le j\le \ell.
$$
Thus, when $g\not\equiv0$, the successive renormalized return maps need not be
small perturbations of a single Lavaurs map. Instead, their phases are changed
by the accumulated corrections
$$
    \sum_{i=0}^{j-1}\mathcal E(z_0,w_i).
$$
There is no general reason for these partial sums to remain bounded, let alone
to be summable. Consequently, even if the first few returns remain in a fixed
compact subset of $\mathcal B_{q_0}$, the phase of the Lavaurs map may drift as
the construction is iterated, which breaks down the trapping argument used in Proposition \ref{prop:main}.
\end{rem}

\section{Example}

As announced in the introduction, we now present an example of a polynomial self-map of $\C^3$ which is not tangent to the identity and nevertheless admits a wandering domain. Moreover, the limit sets of this wandering domain are non-contractible.

By \cite[page 566]{ABT}, the map
$$
P(z,w)=\left(z+z^2,w+w^2+\left(1+\frac{2\pi i}{\ln 2}\right)zw\right)
$$
admits a wandering domain $\Omega\subset\mathbb C^2$. Notice that this form can be obtained from \eqref{2jet} by a simple conjugation using an elementary automorphism.
\medskip

Let $ F(z,w,t)=(P(z,w),t)$ which has a wandering domain $\Omega\times\mathbb C$. Let $S(z,w,t)=(zt,wt,t)$ be an automorphism of $\mathbb C^2\times\mathbb C^*$ and define a map $\widetilde F:=S^{-1}\circ F\circ S$ 
which is given by
$$
    \widetilde F(z,w,t)
    =
    \left(
    z+z^2t,
    w+w^2t+\left(1+\frac{2\pi i}{\log 2}\right)zwt,
    t
    \right).
$$
We define $U:=S^{-1}(\Omega\times\mathbb C^*)$ and observe that it is is a wandering domain for $\widetilde F$. Indeed, $\widetilde F(z,w,0)=(z,w,0)$.

Now let $|\lambda|=1$, with $\lambda$ not a root of unity, and define
$$
    T(z,w,t)=(\lambda z,\lambda w,\lambda^{-1}t),
    \qquad
    G:=T\circ\widetilde F.
$$
Then $T$ commutes with $\widetilde F$, so
$$
    G^n=T^n\circ\widetilde F^n.
$$
We claim that $U$ is also wandering for $G$. Indeed, for $t\neq 0$ we have
$$
    S\circ T(z,w,t)
    =
    S(\lambda z,\lambda w,\lambda^{-1}t)
    =
    (zt,wt,\lambda^{-1}t)
$$
and thus
$$
    R(Z,W,t):= S\circ T\circ S^{-1}(Z,W,t)
    =
    (Z,W,\lambda^{-1}t).
$$
Next observe that
$$
    S\circ G^n
    =
    S\circ T^n\circ\widetilde F^n
    =
    R^n\circ F^n\circ S.
$$
and therefore
$$
    S(G^n(U))
    =
    R^n\bigl(F^n(\Omega\times\mathbb C^*)\bigr)
    =
    R^n\bigl(P^n(\Omega)\times\mathbb C^*\bigr)
    =
    P^n(\Omega)\times\mathbb C^*.
$$
Since $\Omega$ is wandering for $P$, the sets $P^n(\Omega)\times\mathbb C^*$ are pairwise disjoint. By applying $S^{-1}$, we conclude that the sets $G^n(U)$
are pairwise disjoint. Finally, since $G(z,w,0)=(\lambda z, \lambda w,0)$, we can conclude that $U$ is a Fatou component for $G$ and hence a wandering domain.

\medskip

We now describe the limit maps. Let $P^{n_j}|_{\Omega}\to H$ locally
uniformly, where $H=(0,h)$. Passing to a subsequence if necessary, we may
assume that
$$
    \lambda^{-n_j}\to\mu\in S^1.
$$
Then, for $(z,w,t)\in U\subset\C^2\times\C^*$,
$$
    S\circ G^{n_j}(z,w,t)
    =
    \left(
    P^{n_j}(zt,wt),
    \lambda^{-n_j}t
    \right)
    \longrightarrow
    \left(
    H(zt,wt),
    \mu t
    \right).
$$
Hence
$$
    G^{n_j}(z,w,t)
    \longrightarrow
    \left(
    0,
    \frac{h(zt,wt)}{\mu t},
    \mu t
    \right).
$$
In particular, the image of every such limit map contains a copy of
$\mathbb C^*$ in the $t$-direction. Therefore the corresponding limit set is
non-contractible.

\begin{rem}
The map $G$ is a polynomial map of $\C^3$ but it does not lift to an endomorphism of $\mathbb{P}^3$. The existence of an endomorphism with such properties remains an open question. One could try with $\tilde{G}(z,w,t):=G(z,w,t)+(z^3,w^3,t^3)$ with $\lambda$ satisfying a suitable Diophantine condition ensuring that the map $ t\mapsto \lambda^{-1}t+t^3$ is linearizable near the origin. After performing an appropriate non-autonomous conjugation $\tilde{G}$, the problem is reduced to proving the existence of a wandering domain for a non-autonomous sequence of maps in two variables,
$$
\hat{G}_n(z,w)=\varphi_n\circ g_n(z,w),
$$
where
$$
g_n(z,w)=
\left(
z+z^2+a_n(t)z^3,
w+w^2+czw+a_n(t)w^3
\right)
$$
and
$$
\varphi_n(z,w)=b_n(t)\cdot (z,w),
$$
with
$$
b_n(t)=1+O(\lambda^{-n}t),
\qquad
\prod_{n=0}^\infty b_n(t)=1+O(t).
$$
By the results of \cite{ABT}, each map $g_n$ admits a wandering domain. We also strongly believe that these results imply the existence of a wandering domain for the non-autonomous sequence $(g_n)_{n\geq 0}$. To see if the existence of a wandering domain is preserved for the perturbed non-autonomous system $\{\hat{G}_n\}_{n\geq 0}$, one would need to redo the computations of \cite{ABT} to get more precise error estimates. We leave this question to the interested reader.

\end{rem}

\bibliographystyle{amsplain}
\bibliography{bibliography}

\end{document}